\documentclass[11pt]{article}
\usepackage{etex}
\usepackage[all]{xy}

\usepackage{amsfonts}
\usepackage{amsthm}
\usepackage{enumerate}
\usepackage{graphicx}
\usepackage{mathrsfs}
\usepackage{bm}
\usepackage{cite}
\usepackage{amssymb,amsmath} 
\usepackage{makecell}
\usepackage{pgf}
\usepackage{tikz}
\usetikzlibrary{patterns}
\usepackage{pgffor}
\usepackage{pgfcalendar}
\usepackage{pgfpages}
\usepackage{shuffle,yfonts}
\usepackage{mathtools}
\DeclareFontFamily{U}{shuffle}{}
\DeclareFontShape{U}{shuffle}{m}{n}{ <-8>shuffle7 <8->shuffle10}{}

\newcommand{\bfk}{{\boldsymbol{\sl{k}}}}

\newcommand{\de}{\mathrm{d}}
 \allowdisplaybreaks
\usetikzlibrary{arrows,shapes,chains}
 \allowdisplaybreaks

\catcode`!=11
\let\!int\int \def\int{\displaystyle\!int}
\let\!lim\lim \def\lim{\displaystyle\!lim}
\let\!sum\sum \def\sum{\displaystyle\!sum}
\let\!sup\sup \def\sup{\displaystyle\!sup}
\let\!inf\inf \def\inf{\displaystyle\!inf}
\let\!cap\cap \def\cap{\displaystyle\!cap}
\let\!max\max \def\max{\displaystyle\!max}
\let\!min\min \def\min{\displaystyle\!min}
\let\!frac\frac \def\frac{\displaystyle\!frac}
\catcode`!=12

\let\oldsection\section
\renewcommand\section{\setcounter{equation}{0}\oldsection}

\allowdisplaybreaks

\def\a{^{(A)}}

\theoremstyle{plain}
\newtheorem{thm}{Theorem}[section]
\newtheorem{lem}[thm]{Lemma}
\newtheorem{cor}[thm]{Corollary}

\newtheorem{pro}[thm]{Proposition}
\theoremstyle{definition}
\newtheorem{defn}{Definition}[section]
\newtheorem{re}[thm]{Remark}
\newtheorem{exa}[thm]{Example}

\begin{document}
\title{ Regularized Multitangent Functions and Reduction Theorem}
\author{{ {Jia Li\thanks{Email: jialimath001@pku.org.cn}} }\\[1mm]
\small  School of Mathematical Sciences, Peking University, Beijing 100871, P.R. China\\
}

\date{}
\maketitle

\begin{abstract}
    We develop a direct analytic theory of stuffle-regularized
multitangent functions and prove their reduction to finite linear
combinations of monotangent functions, without using mould calculus.
We first establish an asymptotic comparison between one-sided
truncated multiple Hurwitz zeta functions and their stuffle
regularizations at the natural parameter \(T_N-H(s)\), where \(T_N\)
is the harmonic truncation and \(H(s)\) is the harmonic-number function.
Applying this comparison to symmetric multitangent truncations yields
meromorphic, \(1\)-periodic regularized multitangent functions.

An explicit partial-fraction decomposition of each summand, combined
with asymptotic estimates for moving truncation ranges, gives formulas
for the reduction coefficients in terms of stuffle-regularized
multiple zeta values. The constant term and the coefficient of the
monotangent \(\mathcal T(1;s)\) are determined from the limits as
\(\operatorname{Im}s\to\pm\infty\): both vanish whenever the index
contains an entry greater than \(1\), whereas the exceptional indices
\(\{1\}^r\) are evaluated through a sine-quotient generating function.
As a consequence, we obtain a family of relations among regularized
multiple zeta values.
\end{abstract}

\medskip

\noindent{\bf Keywords}: multitangent functions; multiple zeta values; multiple Hurwitz zeta functions; stuffle regularization; partial fractions; reduction theorem.
\medskip

\noindent\textbf{2020 Mathematics Subject Classification.}
Primary 11M32; Secondary 11M99.

\section{Introduction}

Multiple zeta values (MZVs) are the real numbers
\[
\zeta(k_1,\cdots,k_r)
:=
\sum_{0<n_1<\cdots<n_r}
\frac{1}{n_1^{k_1}\cdots n_r^{k_r}},
\qquad
k_1,\cdots,k_{r-1}\geqslant 1,\quad k_r\geqslant 2.
\]
Their study goes back to Euler's investigations of multiple harmonic
sums and was developed systematically in the modern theory of multiple
zeta values by Hoffman, Zagier, and many others
\cite{Euler1,H1992,H1997,DZ1994,Zhao2007d}. Besides their intrinsic interest in number
theory, MZVs occur naturally in the theory of multiple polylogarithms
and mixed Tate motives, as well as in perturbative quantum field theory
\cite{Brown2012,Goncharov2001,BJ,BroadKre1997,Todorov2014,Z2016}.

Two fundamental algebraic structures underlie the theory of MZVs. The
series representation gives rise to the stuffle, or harmonic, product,
whereas the iterated-integral representation gives rise to the shuffle
product. Their interaction produces the double-shuffle relations.
Since these products naturally involve non-admissible indices, a
regularization procedure is required. The regularized double-shuffle
formalism was developed systematically by Ihara, Kaneko, and Zagier
\cite{IKZ2006}; see also \cite{KanekoYa2018} for further identities
involving regularizations.

A functional analogue of the algebra of MZVs was introduced and
studied systematically by Bouillot in the form of multitangent
functions \cite{Bouillot2014}. For an index
\[
\bfk=(k_1,\cdots,k_r)\in(\mathbb Z_{>0})^{\,r}
\]
satisfying \(k_1,k_r>1\), the corresponding multitangent function
is
\[
\mathcal T(\bfk;s)
:=
\sum_{-\infty<n_1<\cdots<n_r<\infty}
\frac{1}
{(n_1+s)^{k_1}\cdots(n_r+s)^{k_r}},
\qquad
s\in\mathbb C-\mathbb Z.
\]
The series is absolutely convergent and defines a \(1\)-periodic
meromorphic function. In depth one, one obtains the monotangent
functions
\[
\mathcal T(k;s)
=
\sum_{n\in\mathbb Z}\frac{1}{(n+s)^k},
\qquad k\geqslant 2,
\]
together with the symmetrically regularized function
\[
\mathcal T(1;s)
:=
\lim_{N\to+\infty}\sum_{-N<n<N}\frac{1}{n+s}
=
\pi\cot(\pi s).
\]

Bouillot developed a comprehensive algebraic and analytic theory of
these functions in the language of mould calculus
\cite{Bouillot2014}. In particular, he introduced the fundamental
operations of reduction into monotangent functions, projection onto
multitangent functions, and trifactorization in terms of multiple
Hurwitz zeta functions. He also studied the regularization of
multitangent functions associated with divergent indices. One of the
main consequences of this theory is the reduction theorem: a
multitangent function can be expressed as a finite linear combination
of monotangent functions, with coefficients belonging to the algebra
of multiple zeta values.

The effectiveness of the multitangent approach is also illustrated by
its applications to relations among MZVs. In particular, Hirose used
Bouillot's theory to obtain an explicit version of the parity theorem
for multiple zeta values \cite{Hirose2025}. 

Although the regularized reduction theorem was first established by Bouillot
using mould calculus, our purpose is different: we provide a direct
analytic proof based only on finite truncations, multiple Hurwitz zeta
functions, partial fraction decomposition, and asymptotic estimates.

The first ingredient is an asymptotic comparison between the
one-sided truncated multiple Hurwitz zeta function
\[
\zeta_{(0,N)}(\bfk;s)
:=
\sum_{0<n_1<\cdots<n_r<N}
\frac{1}
{(n_1+s)^{k_1}\cdots(n_r+s)^{k_r}}
\]
and its stuffle regularization. More precisely, if
\[
T_N:=\sum_{n=1}^{N-1}\frac1n
\]
and
\[
H(s):=\int_0^1\frac{1-t^s}{1-t}\,\mathrm dt,
\]
then we prove
\[
\zeta_{(0,N)}(\bfk;s)
=
\zeta_*^{T_N-H(s)}(\bfk;s)
+
O_{\bfk,s}
\left(
\frac{(1+\log N)^{r-1}}{N}
\right)
\]
locally uniformly in the strip
\[
|\operatorname{Re}s|<\frac12.
\]
The correction \(H(s)\) arises naturally from the elementary
asymptotic formula
\[
\sum_{n=1}^{N-1}\frac1{n+s}
=
T_N-H(s)+O_s(N^{-1}).
\]

Applying this comparison to symmetric finite truncations, we define a
stuffle-regularized multitangent function
\[
\mathcal T_*^T(\bfk;s)
\]
for every index \(\bfk\). We show directly that it is
meromorphic and \(1\)-periodic in \(s\), and that it agrees with the
ordinary multitangent function whenever \(k_1,k_r\geqslant2\).

The second ingredient is a global partial-fraction decomposition of
the rational function
\[
\frac{1}
{(n_1+s)^{k_1}\cdots(n_r+s)^{k_r}}.
\]
For each distinguished variable \(n_j\), its principal part at
\(s=-n_j\) produces powers of the single factor \(n_j+s\), while the
remaining coefficients are finite multiple zeta sums in the
differences \(n_i-n_j\). Applying this identity term by term gives a
finite reduction formula for the symmetric truncation
\(\mathcal T_N(\bfk;s)\). We then analyze the resulting moving
truncation ranges as \(N\to+\infty\). This converts their coefficients
into explicit combinations of stuffle-regularized MZVs.

\begin{thm}[Main Theorem]

Our main result is the following reduction formula. For every
index $\bfk=(k_1,\cdots,k_r)$, one has
\[
\mathcal T_*^T(\bfk;s)
=
c_0(\bfk)
+
c_1(\bfk)\mathcal T(1;s)
+
\sum_{j=1}^{r}
\sum_{m=0}^{k_j-2}
c_{j,m}(\bfk)\,
\mathcal T(k_j-m;s),
\]
where the coefficients \(c_{j,m}(\bfk)\) are given explicitly by
finite sums of products of stuffle-regularized multiple zeta values.
The two exceptional coefficients are
\[
c_0(\bfk)
=
\begin{cases}
\dfrac{(i\pi)^r}{r!},
&
\bfk=\{1\}^r, 2\mid r,
\\[8pt]
0,
&
\text{otherwise},
\end{cases}
\]
and
\[
c_1(\bfk)
=
\begin{cases}
(-1)^{(r-1)/2}\dfrac{\pi^{r-1}}{r!},
&
\bfk=\{1\}^r
,2\nmid r,
\\[8pt]
0,
&
\text{otherwise}.
\end{cases}
\]
Here $\{1\}^r=(\underbrace{1,\cdots,1}_{r})$.
\end{thm}
In particular, if at least one component of \(\bfk\) is greater than
\(1\), both the constant term and the coefficient of
\(\mathcal T(1;s)\) vanish. When \(k_1,k_r\geqslant2\), all the remaining
coefficients are ordinary convergent MZVs. With the normalization \(T=0\), our construction specializes, up to
the notational conventions adopted here, to Bouillot's regularization,
and the reduction formula above recovers his reduction theorem
\cite{Bouillot2014}.

The exceptional pure-one indices are determined by the generating
function
\[
\sum_{r=0}^{\infty}
\mathcal T_*^T(\{1\}^r;s)u^r
=
\frac{\sin\pi(s+u)}{\sin\pi s}.
\]
For all other indices, the constant and simple-pole coefficients are
obtained from the limits of the regularized multitangent function as
\(\operatorname{Im}s\to\pm\infty\). Finally, comparison of the local
Laurent expansion with the reduction formula yields a family of
relations among stuffle-regularized MZVs. The constant-term member of
this family is closely related to the identity used in Hirose's
multitangent proof of the MZV parity theorem \cite{Hirose2025}.

The paper is organized as follows. In Section~2, we introduce
stuffle-regularized multiple Hurwitz zeta functions, prove the
asymptotic comparison with finite truncations, and establish the
estimates required along the imaginary axis. In Section~3, we define
regularized multitangent functions, prove their compatibility with
convergent multitangent functions, and establish their periodicity and
local Laurent expansions. Section~4 contains the global
partial-fraction decomposition and the resulting finite reduction
formula. In Section~5, we study the asymptotic behavior of the shifted
finite multiple zeta coefficients occurring in that formula. Finally,
in Section~6, we prove the regularized reduction theorem, evaluate the
exceptional pure-one indices, and derive the resulting relations among
regularized multiple zeta values.

\section{Stuffle-Regularized Multiple Hurwitz Zeta Functions}

In this section, we introduce multiple Hurwitz zeta functions and
their stuffle regularizations. These functions provide the basic
analytic ingredients for the construction of regularized multitangent
functions in the next section.

We first introduce truncated multiple Hurwitz zeta sums and recall the
standard stuffle-regularization decomposition. We then compare the
one-sided truncation with the corresponding regularized multiple
Hurwitz zeta function at the natural parameter
\[
T_N-H(s),
\]
where \(T_N\) is the harmonic truncation and \(H(s)\) is the harmonic
number function. We also establish a local Taylor expansion and
the decay estimates along the imaginary axis that will be used later
to determine the exceptional terms in the reduction theorem.

\subsection{Indices and truncated multiple Hurwitz zeta functions}

An \emph{index} is a finite sequence of positive integers
\[
\bfk=(k_1,\cdots,k_r)\in(\mathbb Z_{>0})^{\,r}.
\]
Its depth and weight are defined by
\[
\ell(\bfk):=r,
\qquad
|\bfk|:=k_1+\cdots+k_r.
\]
The empty index is denoted by \(\varnothing\), and we set
\[
\ell(\varnothing)=|\varnothing|=0.
\]
The reversal of \(\bfk\) is
\[
\overleftarrow{\bfk}:=(k_r,\cdots,k_1).
\]
For \(0\leqslant j\leqslant r\), we use the notation
\[
\bfk_{[1,j]}:=(k_1,\cdots,k_j),
\qquad
\bfk_{(j,r]}:=(k_{j+1},\cdots,k_r),
\]
with
\[
\bfk_{[1,0]}=\bfk_{(r,r]}=\varnothing.
\]
For \(1\leqslant j\leqslant r\), we also write
\[
\bfk_{[1,j)}:=(k_1,\cdots,k_{j-1}),
\]
with $\bfk_{[1,1)}=\varnothing$. An index \(\bfk\) is called \emph{admissible} if either
\(\bfk=\varnothing\), or its last component satisfies \(k_r\geqslant2\). For convenience, we also define the strip
$$\mathcal{L}:=\left\{s\in\mathbb{C}\bigg||\text{Re}(s)|<\frac{1}{2}\right\}$$

\begin{defn}\label{def:truncated-Hurwitz}
Let
\[
A,B\in\mathbb Z\cup\{-\infty,+\infty\},
\qquad A<B,
\]
and let \(\bfk=(k_1,\cdots,k_r)\) be an index. We define
\[
\zeta_{(A,B)}(\bfk;s)
:=
\sum_{A<n_1<\cdots<n_r<B}
\frac{1}
{(n_1+s)^{k_1}\cdots(n_r+s)^{k_r}},
\]
whenever the sum is defined. If \(B=+\infty\), we require \(k_r\geqslant2\), and if
\(A=-\infty\), we require \(k_1\geqslant2\). For finite \(A\) and \(B\),
the sum is finite. For \(m\in\mathbb Z_{\geqslant0}\), define the shifted coefficient
\begin{align}
\zeta_{(A,B),m}(\bfk;s)
:=
\sum_{\substack{\lambda_1+\cdots+\lambda_r=m\\
                 \lambda_1,\cdots,\lambda_r\geqslant0}}
\left(
\prod_{\nu=1}^r\binom{-k_\nu}{\lambda_\nu}
\right)
\zeta_{(A,B)}
(k_1+\lambda_1,\cdots,k_r+\lambda_r;s).
\label{eq:shifted-truncated-Hurwitz}
\end{align}
For the empty index, we use the convention
\[
\zeta_{(A,B),m}(\varnothing;s):=\delta_{m,0}.
\]
Whenever \(s=0\) is allowed, we abbreviate
\[
\zeta_{(A,B)}(\bfk):=\zeta_{(A,B)}(\bfk;0),
\qquad
\zeta_{(A,B),m}(\bfk):=\zeta_{(A,B),m}(\bfk;0).
\]
In particular, for every admissible index \(\bfk\), we write
\[
\zeta(\bfk;s)
:=
\zeta_{(0,+\infty)}(\bfk;s)
=
\sum_{0<n_1<\cdots<n_r}
\frac{1}
{(n_1+s)^{k_1}\cdots(n_r+s)^{k_r}},
\]
and
\[
\zeta(\bfk):=\zeta(\bfk;0).
\]
\end{defn}

\medskip
\subsection{Stuffle regularization}

We recall the stuffle product on indices. The empty index is the unit,
and recursively
\begin{align}
(a,\boldsymbol{\alpha})*(b,\boldsymbol{\beta})
={}&
\bigl(a,\boldsymbol{\alpha}*(b,\boldsymbol{\beta})\bigr)
+
\bigl(b,(a,\boldsymbol{\alpha})*\boldsymbol{\beta}\bigr)
\nonumber\\
&+
\bigl(a+b,\boldsymbol{\alpha}*\boldsymbol{\beta}\bigr).
\label{eq:stuffle-product}
\end{align}
Let \(\mathfrak H_*^1\) be the stuffle algebra generated by all
indices, and let \(\mathfrak H_*^0\) be the subalgebra generated by the
empty index and all admissible indices.

\begin{lem}[Stuffle decomposition]\label{lem:stuffle-decomposition}
One has
\[
\mathfrak H_*^1=\mathfrak H_*^0[(1)].
\]
More precisely, every index \(\bfk\) of depth \(r\) admits a unique
decomposition
\begin{align}
\bfk
=
\sum_{l=0}^{r}
\boldsymbol{u}_l*(1)^{*l},
\qquad
\boldsymbol{u}_l\in\mathfrak H_*^0,
\label{eq:stuffle-decomposition}
\end{align}
where every index occurring in
\(\boldsymbol{u}_l\) has depth at most \(r-l\).
The decomposition preserves the weight.
\end{lem}

\begin{proof}
This is the standard regularization decomposition of the stuffle
algebra; see, for example, \cite{IKZ2006}. The version used here is
obtained from the usual convention by reversing the indices, since our
admissibility condition is imposed on the last component.
\end{proof}

For an admissible index, the map
\[
\bfk\longmapsto\zeta(\bfk;s)
\]
is a homomorphism for the stuffle product. Lemma
\ref{lem:stuffle-decomposition} therefore gives the following
definition.

\begin{defn}[Stuffle regularization]
\label{def:stuffle-regularization}
Let \(X\) be an indeterminate. If
\[
\bfk
=
\sum_{l=0}^{r}
\boldsymbol{u}_l*(1)^{*l}
\]
is the decomposition in Lemma~\ref{lem:stuffle-decomposition}, define
\begin{align}
\zeta_*^X(\bfk;s)
:=
\sum_{l=0}^{r}
\zeta(\boldsymbol{u}_l;s)X^l,
\label{eq:stuffle-regularized-Hurwitz}
\end{align}
where the convergent Hurwitz zeta map is extended linearly to
\(\mathfrak H_*^0\), and
\[
\zeta(\varnothing;s)=1.
\]
Equivalently, \(\zeta_*^X(\,\cdot\,;s)\) is the unique stuffle
homomorphism extending the convergent multiple Hurwitz zeta map and
satisfying
\[
\zeta_*^X(1;s)=X.
\]
We also put
\[
\zeta_*^X(\bfk):=\zeta_*^X(\bfk;0).
\]
For \(m\geqslant0\), define
\begin{align}
\zeta_{*,m}^X(\bfk)
:=
\sum_{\substack{\lambda_1+\cdots+\lambda_r=m\\
                 \lambda_1,\cdots,\lambda_r\geqslant0}}
\left(
\prod_{\nu=1}^r\binom{-k_\nu}{\lambda_\nu}
\right)
\zeta_*^X
(k_1+\lambda_1,\cdots,k_r+\lambda_r).
\label{eq:shifted-regularized-MZV}
\end{align}
For the empty index, set
\[
\zeta_{*,m}^X(\varnothing):=\delta_{m,0}.
\]
\end{defn}

The decomposition is homogeneous in weight. Consequently,
\(\zeta_*^X(\bfk;s)\) has an expansion of the form
\begin{align}
\zeta_*^X(\bfk;s)
=
\sum_{l=0}^{r}
X^l
\sum_{\boldsymbol{a}}
c_{l,\boldsymbol{a}}
\zeta(\boldsymbol{a};s),
\qquad
c_{l,\boldsymbol{a}}\in\mathbb Q,
\label{eq:regularization-polynomial-form}
\end{align}
where every \(\boldsymbol{a}\) is admissible or empty.

\begin{exa}\label{ex:regularization-21}
The stuffle identity
\[
(2)*(1)=(2,1)+(1,2)+(3)
\]
gives
\[
(2,1)=(2)*(1)-(1,2)-(3).
\]
Therefore
\[
\zeta_*^X(2,1;s)
=
X\cdot\zeta(2;s)-\zeta(1,2;s)-\zeta(3;s),
\]
and, at \(s=0\),
\[
\zeta_*^X(2,1)
=
X\cdot\zeta(2)-\zeta(1,2)-\zeta(3).
\]
\end{exa}

\subsection{The harmonic correction and truncation asymptotics}

\begin{defn}\label{def:harmonic-correction}
For \(\operatorname{Re}s>-1\), we define the \emph{harmonic number function} as
\[
H(s)
:=
\int_0^1\frac{1-t^s}{1-t}\,\de t.
\]
Then
\[
H(s)=\psi(1+s)+\gamma
=
\sum_{n=1}^{\infty}
\left(
\frac1n-\frac1{n+s}
\right),
\]
where \(\psi\) is the digamma function. The function \(H(s)\) extends
meromorphically to \(\mathbb C\) and satisfies
\[
H(s+1)=H(s)+\frac1{s+1}.
\]
\end{defn}

For later use, put
\[
T_N:=\zeta_{(0,N)}(1)
=
\sum_{n=1}^{N-1}\frac1n.
\]

\begin{lem}[Growth of finite Hurwitz sums]
\label{lem:finite-Hurwitz-growth}
Let $\bfk=(k_1,\cdots,k_r)$ be an index. Then, uniformly for \(s\in\mathcal{L}\), one has
\[
\left|
\zeta_{(0,N)}(\bfk;s)
\right|
\leqslant\frac{2^{|\bfk|}}{r!}
(1+\log N)^r.
\]
\end{lem}

\begin{proof}
For \(n\geqslant1\) and \(|\operatorname{Re}s|<1/2\),
\[
|n+s|
\geqslant
|\operatorname{Re}(n+s)|
\geqslant n-\frac12
\geqslant\frac n2.
\]
It follows that
\begin{align*}
\left|
\zeta_{(0,N)}(\bfk;s)
\right|
&\leqslant
2^{|\bfk|}
\sum_{0<n_1<\cdots<n_r<N}
\frac1{n_1\cdots n_r}\\
&\leqslant
\frac{2^{|\bfk|}}{r!}
\left(
\sum_{n=1}^{N-1}\frac1n
\right)^r
\leqslant\frac{2^{|\bfk|}}{r!}
(1+\log N)^r.
\end{align*}
\end{proof}

\begin{thm}[Asymptotic comparison]
\label{thm:regularized-Hurwitz-asymptotic}
Let $\bfk=(k_1,\cdots,k_r)$ be an index. As \(N\to+\infty\), we have
\begin{align}
\zeta_{(0,N)}(\bfk;s)
=
\zeta_*^{T_N-H(s)}(\bfk;s)
+
O_{\bfk,K}
\left(
\frac{(1+\log N)^{r-1}}{N}
\right)
\label{eq:regularized-Hurwitz-asymptotic}
\end{align}
uniformly for \(s\) in every compact subset \(K\subset\mathcal{L}\) .
\end{thm}

\begin{proof}
First,
\begin{align}
\zeta_{(0,N)}(1;s)
=
T_N-H(s)+O_K(N^{-1}),
\label{eq:depth-one-Hurwitz-asymptotic}
\end{align}
uniformly for \(s\in K\). Indeed,
\[
\zeta_{(0,N)}(1;s)-T_N
=
-\sum_{n=1}^{N-1}
\left(
\frac1n-\frac1{n+s}
\right),
\]
and hence
\[
\zeta_{(0,N)}(1;s)-(T_N-H(s))
=
\sum_{n=N}^{\infty}
\frac{s}{n(n+s)}
=
O_K(N^{-1}).
\]
Next, let $\boldsymbol{a}=(a_1,\cdots,a_d)$ be a nonempty admissible index. By
Lemma~\ref{lem:finite-Hurwitz-growth},
\begin{align*}
&
\left|
\zeta(\boldsymbol{a};s)
-
\zeta_{(0,N)}(\boldsymbol{a};s)
\right|\\
&\quad\leqslant
\sum_{n=N}^{\infty}
\frac{
\left|
\zeta_{(0,n)}(a_1,\cdots,a_{d-1};s)
\right|
}{
|n+s|^{a_d}
}\\
&\quad\leqslant\frac{2^{|\boldsymbol{a}|}}{(d-1)!}
\sum_{n=N}^{\infty}
\frac{(1+\log n)^{d-1}}{n^{a_d}}
\leqslant c_{\boldsymbol{a},K}
\frac{(1+\log N)^{d-1}}{N},
\end{align*}
for some constant $c_{\boldsymbol{a},K}>0$, since \(a_d\geqslant2\). Thus
\begin{align}
\zeta_{(0,N)}(\boldsymbol{a};s)
=
\zeta(\boldsymbol{a};s)
+
O_{\boldsymbol{a},K}
\left(
\frac{(1+\log N)^{d-1}}N
\right).
\label{eq:admissible-Hurwitz-tail}
\end{align}
Now take the stuffle decomposition
\[
\bfk
=
\sum_{\ell=0}^{r}
\boldsymbol{u}_l*(1)^{*l}
\]
from Lemma~\ref{lem:stuffle-decomposition}. Since \(\zeta_{(0,N)}(-;s)\) is a
stuffle homomorphism,
\[
\zeta_{(0,N)}(\bfk;s)
=
\sum_{l=0}^{r}
\zeta_{(0,N)}(\boldsymbol{u}_l;s)
(\zeta_{(0,N)}(1;s))^l.
\]
On the other hand,
\[
\zeta_*^{T_N-H(s)}(\bfk;s)
=
\sum_{\ell=0}^{r}
\zeta(\boldsymbol{u}_\ell;s)
\bigl(T_N-H(s)\bigr)^\ell.
\]
Every admissible index occurring in
\(\boldsymbol{u}_\ell\) has depth at most \(r-\ell\).
Combining
\eqref{eq:depth-one-Hurwitz-asymptotic} and
\eqref{eq:admissible-Hurwitz-tail}, and using
\[
T_N-H(s)=O_K(1+\log N),
\qquad
\zeta_{(0,N)}(1;s)=O_K(1+\log N),
\]
we obtain
\[
\zeta_{(0,N)}(\boldsymbol{u}_l;s)
(\zeta_{(0,N)}(1;s))^l
-
\zeta(\boldsymbol{u}_l;s)
\bigl(T_N-H(s)\bigr)^l
=
O_{\bfk,K}
\left(
\frac{(1+\log N)^{r-1}}N
\right).
\]
Summing over \(0\leqslant l\leqslant r\) proves
\eqref{eq:regularized-Hurwitz-asymptotic}.
\end{proof}

\begin{cor}[Negative truncation]
\label{cor:negative-Hurwitz-asymptotic}
Under the assumptions of
Theorem~\ref{thm:regularized-Hurwitz-asymptotic},
\begin{align}
\zeta_{(-N,0)}(\bfk;s)
=
(-1)^{|\bfk|}
\zeta_*^{T_N-H(-s)}
(\overleftarrow{\bfk};-s)
+
O_{\bfk,K}
\left(
\frac{(1+\log N)^{r-1}}N
\right).
\label{eq:negative-Hurwitz-asymptotic}
\end{align}
\end{cor}

\begin{proof}
The change of variables
\[
m_\nu=-n_{r+1-\nu}
\]
gives the exact identity
\[
\zeta_{(-N,0)}(\bfk;s)
=
(-1)^{|\bfk|}
\zeta_{(0,N)}(\overleftarrow{\bfk};-s).
\]
The result therefore follows from
Theorem~\ref{thm:regularized-Hurwitz-asymptotic}.
\end{proof}

\begin{exa}
For \(\bfk=(2,1)\), Example~\ref{ex:regularization-21} and
Theorem~\ref{thm:regularized-Hurwitz-asymptotic} give
\[
\begin{aligned}
\zeta_{(0,N)}(2,1;s)
={}&
\bigl(T_N-H(s)\bigr)\zeta(2;s)
-\zeta(1,2;s)-\zeta(3;s)\\
&+
O_{K}\left(\frac{1+\log N}{N}\right).
\end{aligned}
\]
\end{exa}

\subsection{Local expansion and vertical decay}

\begin{lem}[Local Taylor expansion]
\label{lem:local-Hurwitz-expansion}
Let $\bfk=(k_1,\cdots,k_r)$ be an index. For \(|s|<1/2\),
\begin{align}
\zeta_*^{T-H(s)}(\bfk;s)
=
\sum_{m=0}^{\infty}
\zeta_{*,m}^{T}(\bfk)s^m.
\label{eq:local-Hurwitz-expansion}
\end{align}
The series converges uniformly in \(|s|<1/2\).
\end{lem}

\begin{proof}
Define
\[
\Phi_s^T(\bfk)
:=
\sum_{m=0}^{\infty}
\zeta_{*,m}^{T}(\bfk)s^m.
\]
The Vandermonde identity
\[
\sum_{p+q=m}
\binom{-a}{p}\binom{-b}{q}
=
\binom{-(a+b)}m
\]
shows that \(\Phi_s^T\) is a stuffle homomorphism. If \(\boldsymbol{a}\) is admissible, the binomial expansion gives
\[
\Phi_s^T(\boldsymbol{a})
=
\zeta(\boldsymbol{a};s),
\qquad |s|<1.
\]
Moreover,
\begin{align*}
\Phi_s^T(1)
&=
T+
\sum_{m=1}^{\infty}
(-1)^m\zeta(m+1)s^m\\
&=
T-H(s),
\end{align*}
because
\[
H(s)
=
\sum_{m=1}^{\infty}
(-1)^{m-1}\zeta(m+1)s^m,
\qquad |s|<1.
\]
By the uniqueness in
Definition~\ref{def:stuffle-regularization}, we conclude that
\[
\Phi_s^T(\bfk)
=
\zeta_*^{T-H(s)}(\bfk;s).
\]
This proves \eqref{eq:local-Hurwitz-expansion}.
\end{proof}

\begin{lem}[Decay of admissible Hurwitz zeta functions]
\label{lem:Hurwitz-vertical-decay}
Let $\bfk=(k_1,\cdots,k_r)$ be a nonempty admissible index, and let \(p\in\mathbb{Z}_{\geqslant0}\). Then
\begin{align}
\zeta(\bfk;iy)H(iy)^p
=
O_{\bfk,p}
\left(
|y|^{1-k_r}
(1+\log|y|)^{r+p-1}
\right)
\label{eq:Hurwitz-vertical-decay}
\end{align}
as \(|y|\to\infty\). In particular,
\[
\lim_{y\to\pm\infty}
\zeta(\bfk;iy)H(iy)^p=0.
\]
\end{lem}

\begin{proof}
The classical asymptotic expansion of the digamma function gives
\[
H(iy)=\log(iy)+\gamma+O(|y|^{-1}),
\]
and hence
\[
H(iy)^p=O_p((1+\log|y|)^p).
\]
Separating the last summation variable and using
\[
|n+iy|\geqslant n,
\]
we obtain
\[
|\zeta(\bfk;iy)|
\leqslant c_{\bfk}\cdot
\sum_{n=1}^{\infty}
\frac{(1+\log n)^{r-1}}
{(n^2+|y|^2)^{k_r/2}},
\]
for some constant $c_{\bfk}>0$. If \(1\leqslant n\leqslant |y|\),
\[
(n^2+|y|^2)^{-k_r/2}\leqslant |y|^{-k_r},
\]
so
\[
\sum_{1\leqslant n\leqslant |y|}
\frac{(1+\log n)^{r-1}}
{(n^2+|y|^2)^{k_r/2}}
\leqslant c\cdot
|y|^{1-k_r}(1+\log |y|)^{r-1}.
\]
for some constant $c_1>0$. If \(n>|y|\),
\[
(n^2+|y|^2)^{-k_r/2}\leqslant n^{-k_r},
\]
and therefore
\[
\sum_{n>|y|}
\frac{(1+\log n)^{r-1}}
{(n^2+|y|^2)^{k_r/2}}
\leqslant c_2\cdot
|y|^{1-k_r}(1+\log |y|)^{r-1}.
\]
for some constant $c_2>0$. This proves
\[
\zeta(\bfk;iy)
=
O_{\bfk}
\left(
|y|^{1-k_r}(1+\log |y|)^{r-1}
\right),
\]
and \eqref{eq:Hurwitz-vertical-decay} follows.
\end{proof}

\begin{cor}[Decay of regularized Hurwitz zeta functions]
\label{cor:regularized-Hurwitz-vertical-decay}
Let $
\bfk=(k_1,\cdots,k_r)$ be an index and assume that \(k_j\geqslant2\) for at least one \(j\in\{1,\cdots,r\}\). Then
\[
\lim_{y\to\pm\infty}
\zeta_*^{T-H(iy)}(\bfk;iy)=0.
\]
More precisely, there exists \(A_{\bfk}\geqslant0\) such that
\[
\zeta_*^{T-H(iy)}(\bfk;iy)
=
O_{\bfk}
\left(
\frac{(1+\log|y|)^{A_{\bfk}}}{|y|}
\right).
\]
\end{cor}

\begin{proof}
By \eqref{eq:regularization-polynomial-form},
\[
\zeta_*^X(\bfk;s)
=
\sum_{l,\boldsymbol{a}}
c_{l,\boldsymbol{a}}
X^l\zeta(\boldsymbol{a};s),
\]
where every \(\boldsymbol{a}\) is admissible or empty.

We claim that no empty index occurs in this expansion. Indeed, if
\(\boldsymbol{a}=\varnothing\), homogeneity in the weight would imply
\[
l=|\bfk|.
\]
However, the degree in \(X\) is at most the depth \(r\), whereas the
assumption that some \(k_j\geqslant2\) implies
\[
|\bfk|>r.
\]
This is impossible.

Thus every \(\boldsymbol{a}\) occurring in the expansion is nonempty
and admissible. Substituting \(X=T-H(iy)\), and using
\[
T-H(iy)=O(1+\log|y|),
\]
together with
Lemma~\ref{lem:Hurwitz-vertical-decay}, we obtain
\[
(T-H(iy))^l\zeta(\boldsymbol{a};iy)
=
O_{\bfk}
\left(
|y|^{1-a_{\rm last}}
(1+\log|y|)^{A_{\bfk}}
\right).
\]
Since \(a_{\rm last}\geqslant2\), every term is
\[
O_{\bfk}
\left(
|y|^{-1}(1+\log|y|)^{A_{\bfk}}
\right).
\]
The sum is finite, and the conclusion follows.
\end{proof}

\section{Regularized Multitangent Functions}

In this section, we construct regularized multitangent functions from
the stuffle-regularized multiple Hurwitz zeta functions introduced in
the preceding section.

We begin with symmetric finite truncations and recall the ordinary
multitangent functions associated with convergent indices. By splitting
an integer chain according to the position of \(0\), we obtain an exact
factorization of a finite symmetric truncation into positive and
negative truncated multiple Hurwitz zeta sums. Replacing these
one-sided truncations by their stuffle-regularized asymptotic
expressions motivates the definition of the regularized multitangent
function
\[
\mathcal T_*^T(\bfk;s).
\]
We then prove that this function agrees with the ordinary multitangent
function in the convergent case, is \(1\)-periodic, and admits an
explicit Laurent expansion at every integer.

\subsection{Symmetric truncations and convergent multitangent functions}

\begin{defn}\label{def:symmetric-truncated-multitangent}
Let $\bfk=(k_1,\cdots,k_r)$ be an index. For \(N\in\mathbb Z_{>0}\) and \(s\in\mathbb C-\mathbb Z\),
define the symmetric truncated multitangent function by
\[
\mathcal T_N(\bfk;s)
:=
\sum_{-N<n_1<\cdots<n_r<N}
\frac{1}
{(n_1+s)^{k_1}\cdots(n_r+s)^{k_r}}.
\]
If $k_1,k_r\geqslant2$, we define the convergent multitangent function by
\[
\mathcal T(\bfk;s)
:=
\sum_{-\infty<n_1<\cdots<n_r<\infty}
\frac{1}
{(n_1+s)^{k_1}\cdots(n_r+s)^{k_r}}.
\]
For \(k\geqslant2\), the depth-one functions are
\[
\mathcal T(k;s)
=
\sum_{n\in\mathbb Z}\frac1{(n+s)^k}.
\]
For \(k=1\), we use the symmetric principal value
\[
\mathcal T(1;s)
:=
\lim_{N\to\infty}
\sum_{-N<n<N}\frac1{n+s}
=
\pi\cot(\pi s).
\]
\end{defn}

\begin{pro}\label{pro:convergent-multitangent}
Let $\bfk=(k_1,\ldots,k_r)$ be an index satisfy $k_1,k_r\geqslant2$.
Then the defining series of \(\mathcal T(\bfk;s)\) converges
absolutely and locally uniformly on
\(\mathbb C-\mathbb Z\). Consequently,
\(\mathcal T(\bfk;s)\) is meromorphic on \(\mathbb C\), with poles
contained in \(\mathbb Z\), and
\[
\mathcal T(\bfk;s+1)=\mathcal T(\bfk;s),
\]
for every \(s\in\mathbb C-\mathbb Z\).
\end{pro}

\begin{proof}
The proof is elementary and is left to the interested reader.
\end{proof}

\subsection{The regularized multitangent function}

The following elementary decomposition explains the occurrence of
positive and negative multiple Hurwitz zeta functions.

\begin{lem}[Splitting at the origin]
\label{lem:splitting-at-origin}
For every index $\bfk=(k_1,\cdots,k_r)$, every \(N\geqslant1\), and every \(s\in\mathbb C-\mathbb Z\), one
has
\begin{align}
\mathcal T_N(\bfk;s)
={}&
\sum_{j=0}^{r}
\zeta_{(-N,0)}
\bigl(\bfk_{[1,j]};s\bigr)
\zeta_{(0,N)}
\bigl(\bfk_{(j,r]};s\bigr)
\nonumber\\
&+
\sum_{j=1}^{r}
\frac1{s^{k_j}}
\zeta_{(-N,0)}
\bigl(\bfk_{[1,j)};s\bigr)
\zeta_{(0,N)}
\bigl(\bfk_{(j,r]};s\bigr).
\label{eq:splitting-at-origin}
\end{align}
Here the multiple Hurwitz zeta value of the empty index is understood
to be \(1\).
\end{lem}

\begin{proof}
See \cite{LiCe2025}, proposition 2.2.
\end{proof}

The change of variables
\[
m_\nu=-n_{j+1-\nu}
\]
gives
\[
\zeta_{(-N,0)}
\bigl(\bfk_{[1,j]};s\bigr)
=
(-1)^{|\bfk_{[1,j]}|}
\zeta_{(0,N)}
\bigl(\overleftarrow{\bfk_{[1,j]}};-s\bigr).
\]
This and the asymptotic results of Section~2 motivate the following
definition.

\begin{defn}[Regularized multitangent function]
\label{def:regularized-multitangent}
Let $\bfk=(k_1,\cdots,k_r)$ be an index, and let \(T\) be a regularization parameter. We define
\begin{align}
\mathcal T_*^T(\bfk;s)
:={}&
\sum_{j=0}^{r}
(-1)^{|\bfk_{[1,j]}|}
\zeta_*^{T-H(-s)}
\bigl(\overleftarrow{\bfk_{[1,j]}};-s\bigr)
\zeta_*^{T-H(s)}
\bigl(\bfk_{(j,r]};s\bigr)
\nonumber\\
&+
\sum_{j=1}^{r}
\frac{(-1)^{|\bfk_{[1,j)}|}}{s^{k_j}}
\zeta_*^{T-H(-s)}
\bigl(\overleftarrow{\bfk_{[1,j)}};-s\bigr)
\zeta_*^{T-H(s)}
\bigl(\bfk_{(j,r]};s\bigr).
\label{eq:regularized-multitangent}
\end{align}
The empty-index factors in this formula are understood to be \(1\).

For every fixed \(\bfk\), the function
\(\mathcal T_*^T(\bfk;s)\) is a polynomial in \(T\) whose
coefficients are meromorphic functions of \(s\); in other words,
\[
\mathcal T_*^T(\bfk;s)\in\mathcal M(\mathbb C)[T].
\]
\end{defn}

\begin{re}
With the normalization \(T=0\), Definition
\ref{def:regularized-multitangent} agrees, up to the index and
notation conventions used here, with Bouillot's regularized
multitangent function \cite{Bouillot2014}.
\end{re}

Recall that
\[
T_N:=\sum_{n=1}^{N-1}\frac1n.
\]

\begin{thm}[Asymptotic expansion]
\label{thm:symmetric-truncation-asymptotic}
Let $\bfk=(k_1,\cdots,k_r)$ be an index, then, as \(N\to\infty\),
\begin{align}
\mathcal T_N(\bfk;s)
=
\mathcal T_*^{T_N}(\bfk;s)
+
O_{\bfk,K}
\left(
\frac{(1+\log N)^{r-1}}N
\right)
\label{eq:symmetric-truncation-asymptotic}
\end{align}
locally uniformly for \(s\) in every compact subset $K\subset\mathcal{L}$.
\end{thm}

\begin{proof}
Starting from Lemma~\ref{lem:splitting-at-origin}, apply
Theorem~\ref{thm:regularized-Hurwitz-asymptotic} to every positive
factor and Corollary~\ref{cor:negative-Hurwitz-asymptotic} to every
negative factor.

For example,
\[
\begin{aligned}
\zeta_{(-N,0)}
\bigl(\bfk_{[1,j]};s\bigr)
={}&
(-1)^{|\bfk_{[1,j]}|}
\zeta_*^{T_N-H(-s)}
\bigl(\overleftarrow{\bfk_{[1,j]}};-s\bigr)\\
&+
O_{\bfk,K}
\left(
\frac{(1+\log N)^{j-1}}N
\right),
\end{aligned}
\]
and
\[
\begin{aligned}
\zeta_{(0,N)}
\bigl(\bfk_{(j,r]};s\bigr)
={}&
\zeta_*^{T_N-H(s)}
\bigl(\bfk_{(j,r]};s\bigr)\\
&+
O_{\bfk,K}
\left(
\frac{(1+\log N)^{r-j-1}}N
\right).
\end{aligned}
\]

A regularized Hurwitz zeta value associated with an index of depth
\(d\) grows at most like \(O_{\bfk,K}((1+\log N)^d)\) when its
regularization parameter is \(T_N-H(\pm s)\). Hence the error in the
product of a depth-\(j\) factor and a depth-\((r-j)\) factor is
\[
O_{\bfk,K}
\left(
\frac{(1+\log N)^{r-1}}N
\right).
\]
The terms in the second sum of
\eqref{eq:splitting-at-origin} contain total depth \(r-1\), so their
errors satisfy the same bound, or a stronger one. Since only finitely
many values of \(j\) occur, summing all these estimates gives
\eqref{eq:symmetric-truncation-asymptotic}.
\end{proof}

\begin{cor}[Compatibility with the convergent case]
\label{cor:compatibility-convergent-multitangent}
Let $\bfk=(k_1,\cdots,k_r)$ be an index, if $k_1,k_r\geqslant2$, then
\[
\mathcal T_*^T(\bfk;s)=\mathcal T(\bfk;s)
\]
for every \(T\) and every \(s\in\mathcal{L}\).
In particular, the left-hand side is independent of \(T\).
\end{cor}

\begin{proof}
Fix \(s\in\mathcal{L}\), and put
\[
P_s(X):=\mathcal T_*^X(\bfk;s)\in\mathbb C[X].
\]
By Proposition~\ref{pro:convergent-multitangent},
\[
\mathcal T_N(\bfk;s)\longrightarrow\mathcal T(\bfk;s).
\]
On the other hand,
Theorem~\ref{thm:symmetric-truncation-asymptotic} gives
\[
P_s(T_N)
=
\mathcal T_N(\bfk;s)+o(1).
\]
Thus \(P_s(T_N)\) has a finite limit as \(N\to\infty\). Since
\[
T_N\longrightarrow+\infty,
\]
this is possible only if the polynomial \(P_s\) is constant. Its
constant value must be \(\mathcal T(\bfk;s)\). Hence
\[
\mathcal T_*^T(\bfk;s)=\mathcal T(\bfk;s)
\]
for every \(T\).
\end{proof}

\subsection{Periodicity}

\begin{pro}\label{pro:regularized-periodicity}
For every index \(\bfk\), we have
\[\boxed{
\mathcal T_*^T(\bfk;s+1)
=
\mathcal T_*^T(\bfk;s)}
\]
for all \(s\in\mathbb C-\mathbb Z\).
\end{pro}

\begin{proof}
Fix \(s\in\mathbb C-\mathbb Z\), by Theorem~\ref{thm:symmetric-truncation-asymptotic},
\begin{align}
\mathcal T_N(\bfk;s)
=
\mathcal T_*^{T_N}(\bfk;s)
+
O_{\bfk,s}
\left(
\frac{(1+\log N)^{r-1}}N
\right).
\label{eq:periodicity-first-asymptotic}
\end{align}
Now make the change of variables
\[
m_i=n_i-1.
\]
Then
\[
\mathcal T_N(\bfk;s)
=
\sum_{-N-1<m_1<\cdots<m_r<N-1}
\frac1{(m_1+s+1)^{k_1}\cdots(m_r+s+1)^{k_r}}.
\]
Splitting this asymmetric truncation at \(0\), and applying the
one-sided asymptotic formulas with truncation parameters \(N+1\) and
\(N-1\), gives
\begin{align}
\mathcal T_N(\bfk;s)
=
\mathcal T_*^{T_N}(\bfk;s+1)
+
O_{\bfk,s}
\left(
\frac{(1+\log N)^{r-1}}N
\right).
\label{eq:periodicity-second-asymptotic}
\end{align}

We justify the use of the common parameter \(T_N\). Indeed,
\[
T_{N+1}-T_N=\frac1N,
\qquad
T_{N-1}-T_N=-\frac1{N-1}.
\]
For an index of depth \(d\), the regularized Hurwitz zeta function is
a polynomial of degree at most \(d\) in its regularization parameter.
Consequently, replacing \(T_{N+1}\) or \(T_{N-1}\) by \(T_N\)
changes such a factor by at most
\[
O_{\bfk,s}
\left(
\frac{(1+\log N)^{d-1}}N
\right),
\]
which is absorbed by the error in
\eqref{eq:periodicity-second-asymptotic}.

Subtracting
\eqref{eq:periodicity-first-asymptotic} and
\eqref{eq:periodicity-second-asymptotic}, we obtain
\[
\mathcal T_*^{T_N}(\bfk;s)-\mathcal T_*^{T_N}(\bfk;s+1)
=
O_{\bfk,s}
\left(
\frac{(1+\log N)^{r-1}}N
\right)
=o(1).
\]
The function
\[
Q_s(X):=\mathcal T_*^X(\bfk;s)-\mathcal T_*^X(\bfk;s+1)
\]
is a polynomial in \(X\), and \(T_N\to+\infty\). A nonzero
polynomial cannot tend to \(0\) along an unbounded sequence.
Therefore
\[
Q_s(X)\equiv0.
\]
Hence
\[
\mathcal T_*^T(\bfk;s+1)
=
\mathcal T_*^T(\bfk;s)
\]
for every \(T\).
\end{proof}

\subsection{Local Laurent expansion}

For \(m\geqslant0\), define
\begin{align}
A_m^T(\bfk)
:=
\sum_{j=0}^{r}
(-1)^{|\bfk_{[1,j]}|}
\sum_{\substack{a+b=m\\a,b\geqslant0}}
(-1)^a
\zeta_{*,a}^T
\bigl(\overleftarrow{\bfk_{[1,j]}}\bigr)
\zeta_{*,b}^T
\bigl(\bfk_{(j,r]}\bigr),
\label{eq:def-Am}
\end{align}
and, for \(1\leqslant j\leqslant r\),
\begin{align}
B_{j,m}^T(\bfk)
:=
(-1)^{|\bfk_{[1,j)}|}
\sum_{\substack{a+b=m\\a,b\geqslant0}}
(-1)^a
\zeta_{*,a}^T
\bigl(\overleftarrow{\bfk_{[1,j)}}\bigr)
\zeta_{*,b}^T
\bigl(\bfk_{(j,r]}\bigr).
\label{eq:def-Bjm}
\end{align}

\begin{lem}[Local Laurent expansion]
\label{lem:local-multitangent-expansion}
For every index \(\bfk\) and every
\[
0<|s|<\frac12,
\]
one has
\begin{align}
\mathcal T_*^T(\bfk;s)
=
\sum_{m=0}^{\infty}A_m^T(\bfk)s^m
+
\sum_{j=1}^{r}
\sum_{m=0}^{\infty}
B_{j,m}^T(\bfk)s^{m-k_j}.
\label{eq:local-multitangent-expansion}
\end{align}
The series converges locally uniformly on the punctured disc
\[
0<|s|<\frac12.
\]
\end{lem}

\begin{proof}
By Lemma~\ref{lem:local-Hurwitz-expansion},
\[
\zeta_*^{T-H(s)}(\boldsymbol{a};s)
=
\sum_{b=0}^{\infty}
\zeta_{*,b}^T(\boldsymbol{a})s^b,
\]
whereas replacing \(s\) by \(-s\) gives
\[
\zeta_*^{T-H(-s)}(\boldsymbol{a};-s)
=
\sum_{a=0}^{\infty}
(-1)^a
\zeta_{*,a}^T(\boldsymbol{a})s^a.
\]

Substituting these two expansions into
Definition~\ref{def:regularized-multitangent}, and applying the Cauchy
product formula, the first sum in
\eqref{eq:regularized-multitangent} becomes
\[
\sum_{m=0}^{\infty}A_m^T(\bfk)s^m,
\]
while its \(j\)-th term in the second sum becomes
\[
\sum_{m=0}^{\infty}
B_{j,m}^T(\bfk)s^{m-k_j}.
\]
Summing over \(j\) proves
\eqref{eq:local-multitangent-expansion}.

The power series from
Lemma~\ref{lem:local-Hurwitz-expansion} converge locally uniformly for
\(|s|<1/2\). Since all sums over \(j\) are finite, the corresponding
Cauchy products converge locally uniformly on the punctured disc.
\end{proof}

\begin{cor}[Principal part at the origin]
\label{cor:principal-part-multitangent}
The principal part of \(\mathcal T_*^T(\bfk;s)\) at \(s=0\) is
\begin{align}
\operatorname{PP}_{s=0}
\mathcal T_*^T(\bfk;s)
=
\sum_{j=1}^{r}
\sum_{m=0}^{k_j-1}
B_{j,m}^T(\bfk)s^{m-k_j}.
\label{eq:principal-part-multitangent}
\end{align}
In particular, the pole order at every integer is at most $\max_{1\leqslant j\leqslant r}\{k_j\}$.
\end{cor}

\begin{proof}
In the first sum of
\eqref{eq:local-multitangent-expansion}, all powers of \(s\) are
nonnegative. In the second sum, a negative power occurs precisely
when $m<k_j$. This gives \eqref{eq:principal-part-multitangent}. The assertion at
every integer follows from
Proposition~\ref{pro:regularized-periodicity}.
\end{proof}

\section{Finite Partial-Fraction Reduction}

In this section, we establish the finite reduction formula for
multitangent functions by means of a global partial-fraction
decomposition. The main idea is to interpret the reduction as a consequence of the partial fraction decomposition of a rational function.

\begin{lem}\label{lem:partial-fraction}
Let \(n_1,\cdots,n_r\in\mathbb{Z}\) be pairwise distinct numbers and $\bfk=(k_1,\cdots,k_r)$ be an index. Then, for every $s\in\mathbb C-\mathbb{Z}$, we have

\begin{equation}
\begin{split}
&\frac{1}{(n_1+s)^{k_1}\cdots(n_r+s)^{k_r}}\\
&=\sum_{j=1}^{r}\sum_{m_j=0}^{k_j-1}\left(\sum_{\substack{l_1+\cdots+l_{j-1}+l_{j+1}+\cdots+l_r=m_j\\ l_1,\cdots,l_{j-1},l_{j+1},\cdots,l_r\geqslant 0}}\prod_{\substack{1\leqslant i\leqslant r\\i\neq j}}\frac{\binom{-k_i}{l_i}}{(n_i-n_j)^{k_i+l_i}}\right)\frac{1}{(s+n_j)^{k_j-m_j}}.
\end{split}
\label{eq:partial-fraction}
\end{equation}

\end{lem}

\begin{proof}
Set
\[
F(s):=
\frac{1}
{(s+n_1)^{k_1}\cdots(s+n_r)^{k_r}}.
\]
If \(r=1\), the assertion is immediate. Hence we may assume
\(r\geqslant2\). We first compute the principal part of \(F(s)\) at each pole
\(s=-n_j\). Fix \(j\in\{1,\cdots,r\}\), then
\[
s+n_i=s+n_j+(n_i-n_j)
\qquad (i\ne j),
\]
and hence
\[
F(s)
=
\frac{1}{(s+n_j)^{k_j}}
\prod_{\substack{1\leqslant i\leqslant r\\i\neq j}}
\frac{1}{\bigl(s+n_j+(n_i-n_j)\bigr)^{k_i}}.
\]
Since the integers \(n_1,\cdots,n_r\) are pairwise distinct, we have
\(n_i-n_j\ne0\) for every \(i\ne j\). Therefore, for
\[
|s+n_j|<
\min_{i\ne j}|n_i-n_j|,
\]
the generalized binomial expansion gives
\begin{align*}
\frac{1}{\bigl(s+n_j+(n_i-n_j)\bigr)^{k_i}}
&=
\frac{1}{(n_i-n_j)^{k_i}}
\left(
1+\frac{s+n_j}{n_i-n_j}
\right)^{-k_i} =
\sum_{l_i=0}^{\infty}
\binom{-k_i}{l_i}
\frac{(s+n_j)^{l_i}}
{(n_i-n_j)^{k_i+l_i}}.
\end{align*}
Multiplying these absolutely convergent power series, we obtain
\begin{align*}
\prod_{\substack{1\leqslant i\leqslant r\\i\ne j}}
\frac{1}{\bigl(s+n_j+n_i-n_j\bigr)^{k_i}}
&=
\sum_{m_j=0}^{\infty}
\left(
\sum_{\substack{l_1+\cdots+l_{j-1}+l_{j+1}+\cdots+l_r=m_j\\ l_1,\cdots,l_{j-1},l_{j+1},\cdots,l_r\geqslant 0}}
\prod_{\substack{1\leqslant i\leqslant r\\i\ne j}}
\frac{\binom{-k_i}{l_i}}
{(n_i-n_j)^{k_i+l_i}}
\right)(s+n_j)^{m_j}.
\end{align*}
Consequently,
\begin{align*}
F(s)
&=\sum_{m_j=0}^{\infty}
\left(
\sum_{\substack{l_1+\cdots+l_{j-1}+l_{j+1}+\cdots+l_r=m_j\\ l_1,\cdots,l_{j-1},l_{j+1},\cdots,l_r\geqslant 0}}
\prod_{\substack{1\leqslant i\leqslant r\\i\ne j}}
\frac{\binom{-k_i}{l_i}}
{(n_i-n_j)^{k_i+l_i}}
\right)(s+n_j)^{m_j-k_j}.
\end{align*}
The negative powers of \(s+n_j\) occur precisely when $0\leqslant m_j\leqslant k_j-1$. Therefore, the principal part of \(F(s)\) at \(s=-n_j\) is
\begin{align}
\operatorname{PP}_{s=-n_j}F(s)
&=
\sum_{m_j=0}^{k_j-1}
\left(
\sum_{\substack{l_1+\cdots+l_{j-1}+l_{j+1}+\cdots+l_r=m_j\\ l_1,\cdots,l_{j-1},l_{j+1},\cdots,l_r\geqslant 0}}
\prod_{\substack{1\leqslant i\leqslant r\\i\ne j}}
\frac{\binom{-k_i}{l_i}}
{(n_i-n_j)^{k_i+l_i}}
\right)
\frac{1}{(s+n_j)^{k_j-m_j}}.
\label{eq:principal-part}
\end{align}

Now let \(R(s)\) denote the right-hand side of
\eqref{eq:partial-fraction}. By \eqref{eq:principal-part}, the
principal part of \(R(s)\) at \(s=-n_j\) is exactly the principal
part of \(F(s)\) there. It follows that
\[
D(s):=F(s)-R(s)
\]
has removable singularities at every point
\[
-n_1,\cdots,-n_r.
\]
After removing these singularities, \(D(s)\) is an entire function.

Moreover, as \(|s|\to\infty\),
\[
F(s)
=
O\left(
|s|^{-|\bfk|}
\right),
\]
while every summand in \(R(s)\) is \(O(|s|^{-1})\). Hence
\[
D(s)=O(|s|^{-1}),
\qquad |s|\to\infty.
\]
In particular,
\[
D(s)\longrightarrow0
\qquad (|s|\to\infty).
\]

Thus \(D\) is an entire bounded function. By Liouville's theorem,
\(D\) is constant; since \(D(s)\to0\) at infinity, this constant is
zero. Therefore
\[
D(s)\equiv0,
\]
which proves \eqref{eq:partial-fraction}.
\end{proof}

\begin{thm}[Finite Reduction Formula]\label{thm:finite_pf_reduction}
Let $
\bfk=(k_1,\cdots,k_r)$ be an index. For \(s\in\mathbb C-\mathbb Z\), one has
\begin{align}
\mathcal T_N(\bfk;s)
={}&
\sum_{j=1}^r
\sum_{m_j=0}^{k_j-1}
\sum_{-N<n_j<N}
\left(
\sum_{\substack{a+b=m_j\\a,b\geqslant0}}
\zeta_{(-N-n_j,0),a}\bigl(\bfk_{[1,j)}\bigr)
\zeta_{(0,N-n_j),b}\bigl(\bfk_{(j,r]}\bigr)
\right)
\frac{1}{(n_j+s)^{k_j-m_j}}.
\label{eq:finite-pf-reduction}
\end{align}
\end{thm}

\begin{proof}
By definition,
\[
\mathcal T_N(\bfk;s)
=
\sum_{-N<n_1<\cdots<n_r<N}
\frac{1}
{(n_1+s)^{k_1}\cdots(n_r+s)^{k_r}}.
\]
For each fixed tuple
\[
-N<n_1<\cdots<n_r<N,
\]
the integers \(n_1,\cdots,n_r\) are pairwise distinct. Hence by Lemma \ref{lem:partial-fraction}, the
partial-fraction decomposition gives
\begin{align}
&\frac{1}{(n_1+s)^{k_1}\cdots(n_r+s)^{k_r}}\nonumber\\
&=\sum_{j=1}^{r}\sum_{m_j=0}^{k_j-1}\left(\sum_{\substack{l_1+\cdots+l_{j-1}+l_{j+1}+\cdots+l_r=m_j\\ l_1,\cdots,l_{j-1},l_{j+1},\cdots,l_r\geqslant 0}}\prod_{\substack{1\leqslant i\leqslant r\\i\neq j}}\frac{\binom{-k_i}{l_i}}{(n_i-n_j)^{k_i+l_i}}\right)\frac{1}{(s+n_j)^{k_j-m_j}}.
\label{eq:partial-fraction-used}
\end{align}
Substituting \eqref{eq:partial-fraction-used} into the definition of
\(\mathcal T_N(\bfk;s)\), and interchanging the finite sums, we obtain
\begin{align}
\mathcal T_N(\bfk;s)
={}&
\sum_{j=1}^{r}
\sum_{m_j=0}^{k_j-1}
\sum_{-N<n_j<N}
\frac{C_{j,m_j}(n_j)}
{(n_j+s)^{k_j-m_j}},
\label{eq:TN-Cjm}
\end{align}
where
\begin{align}
C_{j,m_j}(n_j)
:={}&
\sum_{\substack{
-N<n_1<\cdots<n_{j-1}<n_j\\
n_j<n_{j+1}<\cdots<n_r<N
}}
\;
\sum_{\substack{l_1+\cdots+l_{j-1}+l_{j+1}+\cdots+l_r=m_j\\ l_1,\cdots,l_{j-1},l_{j+1},\cdots,l_r\geqslant 0}}\prod_{\substack{1\leqslant i\leqslant r\\i\neq j}}\frac{\binom{-k_i}{l_i}}{(n_i-n_j)^{k_i+l_i}}.
\label{eq:Cjm-definition}
\end{align}
We now calculate \(C_{j,m_j}(n_j)\). Decompose $m_j=a+b$, where
\[
a=l_1+\cdots+l_{j-1},
\qquad
b=l_{j+1}+\cdots+l_r.
\]
After grouping the summation according to \(a\) and \(b\), the left
and right variables become independent. Thus
\begin{align}
C_{j,m_j}(n_j)
={}&
\sum_{\substack{a+b=m_j\\a,b\geqslant0}}
L_{j,a}(n_j)R_{j,b}(n_j),
\label{eq:C-factorization}
\end{align}
where
\begin{align}
L_{j,a}(n_j)
:={}&
\sum_{\substack{
l_1+\cdots+l_{j-1}=a\\
l_1,\cdots,l_{j-1}\geqslant0
}}
\left(
\prod_{i=1}^{j-1}\binom{-k_i}{l_i}
\right)
\sum_{-N<n_1<\cdots<n_{j-1}<n_j}
\prod_{i=1}^{j-1}
\frac{1}{(n_i-n_j)^{k_i+l_i}},
\label{eq:left-part}
\\
R_{j,b}(n_j)
:={}&
\sum_{\substack{
l_{j+1}+\cdots+l_r=b\\
l_{j+1},\cdots,l_r\geqslant0
}}
\left(
\prod_{i=j+1}^{r}\binom{-k_i}{l_i}
\right)
\sum_{n_j<n_{j+1}<\cdots<n_r<N}
\prod_{i=j+1}^{r}
\frac{1}{(n_i-n_j)^{k_i+l_i}}.
\label{eq:right-part}
\end{align}
For the left-hand factor, make the change of variables
\[
u_i=n_i-n_j,
\qquad 1\leqslant i<j.
\]
The inequalities
\[
-N<n_1<\cdots<n_{j-1}<n_j
\]
become
\[
-N-n_j<u_1<\cdots<u_{j-1}<0.
\]
Consequently,
\begin{align}
L_{j,a}(n_j)
={}&
\sum_{\substack{
l_1+\cdots+l_{j-1}=a\\
l_1,\cdots,l_{j-1}\geqslant0
}}
\left(
\prod_{i=1}^{j-1}\binom{-k_i}{l_i}
\right)
\zeta_{(-N-n_j,0)}
(k_1+l_1,\cdots,k_{j-1}+l_{j-1})
\nonumber\\
={}&
\zeta_{(-N-n_j,0),a}(k_1,\cdots,k_{j-1})
\nonumber\\
={}&
\zeta_{(-N-n_j,0),a}\bigl(\bfk_{[1,j)}\bigr).
\label{eq:left-zeta}
\end{align}
Similarly, for the right-hand factor, put
\[
v_i=n_i-n_j,
\qquad j<i\leqslant r.
\]
Then
\[
n_j<n_{j+1}<\cdots<n_r<N
\]
is equivalent to
\[
0<v_{j+1}<\cdots<v_r<N-n_j.
\]
It follows that
\begin{align}
R_{j,b}(n_j)
={}&
\sum_{\substack{
l_{j+1}+\cdots+l_r=b\\
l_{j+1},\cdots,l_r\geqslant0
}}
\left(
\prod_{i=j+1}^{r}\binom{-k_i}{l_i}
\right)
\zeta_{(0,N-n_j)}
(k_{j+1}+l_{j+1},\cdots,k_r+l_r)
\nonumber\\
={}&
\zeta_{(0,N-n_j),b}(k_{j+1},\cdots,k_r)
\nonumber\\
={}&
\zeta_{(0,N-n_j),b}\bigl(\bfk_{(j,r]}\bigr).
\label{eq:right-zeta}
\end{align}
Combining \eqref{eq:C-factorization},
\eqref{eq:left-zeta}, and \eqref{eq:right-zeta}, we obtain
\begin{align}
C_{j,m_j}(n_j)
=
\sum_{\substack{a+b=m_j\\a,b\geqslant0}}
\zeta_{(-N-n_j,0),a}\bigl(\bfk_{[1,j)}\bigr)
\zeta_{(0,N-n_j),b}\bigl(\bfk_{(j,r]}\bigr).
\label{eq:C-final}
\end{align}
Finally, substituting \eqref{eq:C-final} into
\eqref{eq:TN-Cjm} yields
\begin{align*}
\mathcal T_N(\bfk;s)
={}&
\sum_{j=1}^r
\sum_{m_j=0}^{k_j-1}
\sum_{-N<n_j<N}
\left(
\sum_{\substack{a+b=m_j\\a,b\geqslant0}}
\zeta_{(-N-n_j,0),a}\bigl(\bfk_{[1,j)}\bigr)
\zeta_{(0,N-n_j),b}\bigl(\bfk_{(j,r]}\bigr)
\right)
\frac{1}{(n_j+s)^{k_j-m_j}},
\end{align*}
which is exactly \eqref{eq:finite-pf-reduction}.
\end{proof}

\section{Moving-Window Asymptotics}

In the finite partial-fraction reduction obtained in the preceding
section, the coefficients are products of shifted finite multiple zeta
values whose truncation bounds depend on the distinguished summation
variable. More precisely, they involve expressions of the form
\[
\zeta_{(-N-n,0),a}(\boldsymbol{\alpha})
\zeta_{(0,N-n),b}(\boldsymbol{\beta}),
\qquad -N<n<N.
\]
The purpose of this section is to determine the asymptotic behavior of
the corresponding one-dimensional sums as \(N\to\infty\).

For poles of order at least \(2\), the moving coefficients may be
replaced by their values at \(n=0\). For the simple pole, an additional
symmetric moving-window sum remains. We show that this extra term is
given asymptotically by an explicit logarithmic integral involving
stuffle-regularized multiple zeta values. Combining these results with
the finite reduction formula yields an asymptotic reduction of
\(\mathcal T_N(\bfk;s)\) into monotangent functions.

Throughout this section, we put
\[
T_N:=\zeta_{(0,N)}(1).
\]

\subsection{Basic estimates for shifted finite multiple zeta values}

\begin{lem}[Growth and local-shift estimates]
\label{lem:shifted-truncation-estimates}
Let $
\boldsymbol{\alpha}
=(\alpha_1,\cdots,\alpha_d)
\in(\mathbb Z_{>0})^{\,d}, a\in\mathbb Z_{\geqslant0}$. Then the following assertions hold.

First,
\begin{align}
\left|
\zeta_{(0,M),a}(\boldsymbol{\alpha})
\right|
\leqslant c_1(\alpha,a)\cdot
(1+\log M)^d.
\label{eq:shifted-truncation-growth}
\end{align}

Second,
\begin{align}
\left|
\zeta_{(0,M+1),a}(\boldsymbol{\alpha})
-
\zeta_{(0,M),a}(\boldsymbol{\alpha})
\right|
\leqslant c_2(\alpha,a)\cdot
\frac{(1+\log M)^{d-1}}{M}.
\label{eq:shifted-one-step}
\end{align}

Consequently, for every integer \(h\) satisfying
\[
|h|\leqslant \frac M2,
\]
one has
\begin{align}
\left|
\zeta_{(0,M+h),a}(\boldsymbol{\alpha})
-
\zeta_{(0,M),a}(\boldsymbol{\alpha})
\right|
\leqslant c_3(\alpha,a)\cdot
\frac{|h|}{M}(1+\log M)^{d-1}.
\label{eq:shifted-local-shift}
\end{align}

Finally,
\begin{align}
\zeta_{(-M,0),a}(\boldsymbol{\alpha})
=
(-1)^{|\boldsymbol{\alpha}|+a}
\zeta_{(0,M),a}
(\overleftarrow{\boldsymbol{\alpha}}).
\label{eq:shifted-negative-reversal}
\end{align}
\end{lem}

\begin{proof}
By definition,
\[
\zeta_{(0,M),a}(\boldsymbol{\alpha})
=
\sum_{\substack{
\lambda_1+\cdots+\lambda_d=a\\
\lambda_1,\cdots,\lambda_d\geqslant0
}}
\left(
\prod_{\nu=1}^{d}
\binom{-\alpha_\nu}{\lambda_\nu}
\right)
\zeta_{(0,M)}
(\alpha_1+\lambda_1,\cdots,\alpha_d+\lambda_d).
\]
Since every exponent \(\alpha_\nu+\lambda_\nu\) is at least \(1\),
\[
\left|
\zeta_{(0,M)}
(\alpha_1+\lambda_1,\cdots,\alpha_d+\lambda_d)
\right|
\leqslant
\frac1{d!}
\left(
\sum_{n=1}^{M-1}\frac1n
\right)^d\leqslant(1+\log M)^d.
\]
The sum over \((\lambda_1,\cdots,\lambda_d)\) is finite, which proves
\eqref{eq:shifted-truncation-growth}.

The new terms appearing when the upper bound changes from \(M\) to
\(M+1\) are precisely those for which the last summation variable is
equal to \(M\). Hence
\begin{align*}
&\zeta_{(0,M+1),a}(\boldsymbol{\alpha})
-
\zeta_{(0,M),a}(\boldsymbol{\alpha})
\\
&=
\sum_{\substack{
\lambda_1+\cdots+\lambda_d=a\\
\lambda_1,\cdots,\lambda_d\geqslant0
}}
\left(
\prod_{\nu=1}^{d}
\binom{-\alpha_\nu}{\lambda_\nu}
\right)
\frac{
\zeta_{(0,M)}
(\alpha_1+\lambda_1,\cdots,
 \alpha_{d-1}+\lambda_{d-1})
}{
M^{\alpha_d+\lambda_d}
}.
\end{align*}
Here the numerator is interpreted as \(1\) when \(d=1\). Since
\[
\alpha_d+\lambda_d\geqslant1,
\]
the growth estimate for the prefix gives
\eqref{eq:shifted-one-step}.

If \(h>0\), summing \eqref{eq:shifted-one-step} from \(M\) to
\(M+h-1\) gives \eqref{eq:shifted-local-shift}. If \(h<0\), we sum
from \(M+h\) to \(M-1\). Since \(|h|\leqslant M/2\), all intermediate
indices are comparable with \(M\), and the same estimate follows.

For the final identity, changing the negative summation variables into
positive ones and reversing their order gives
\[
\begin{aligned}
&
\zeta_{(-M,0)}
(\alpha_1+\lambda_1,\cdots,\alpha_d+\lambda_d)
=
(-1)^{|\boldsymbol{\alpha}|+\lambda_1+\cdots+\lambda_d}
\zeta_{(0,M)}
(\alpha_d+\lambda_d,\cdots,\alpha_1+\lambda_1).
\end{aligned}
\]
Since
\[
\lambda_1+\cdots+\lambda_d=a,
\]
summing over the \(\lambda_\nu\)'s proves
\eqref{eq:shifted-negative-reversal}.
\end{proof}

\begin{lem}[Regularized polynomial approximation]
\label{lem:shifted-regularized-approximation}
Let \(\boldsymbol{\alpha}\) be a fixed index and let \(a\geqslant0\).
There exists \(A\geqslant0\), depending only on
\(\boldsymbol{\alpha}\) and \(a\), such that
\begin{align}
\zeta_{(0,M),a}(\boldsymbol{\alpha})
=
\zeta_{*,a}^{T_M}(\boldsymbol{\alpha})
+
O_{\boldsymbol{\alpha},a}
\left(
\frac{(1+\log M)^A}{M}
\right),
\label{eq:shifted-approximation-TM}
\end{align}
where $T_M:=\zeta_{(0,M)}(1)=\log M+\gamma+O(M^{-1})$. Then one also has
\begin{align}
\zeta_{(0,M),a}(\boldsymbol{\alpha})
=
\zeta_{*,a}^{\log M+\gamma}(\boldsymbol{\alpha})
+
O_{\boldsymbol{\alpha},a}
\left(
\frac{(1+\log M)^A}{M}
\right).
\label{eq:shifted-approximation-LM}
\end{align}
Moreover,
\begin{align}
\left|
\zeta_{*,a}^{\log M+\gamma}(\boldsymbol{\alpha})
\right|
\leqslant c(\boldsymbol{\alpha},a)\cdot
(1+\log M)^A,
\label{eq:shifted-polynomial-growth}
\end{align}
for some constant $c(\boldsymbol{\alpha},a)>0$.
\end{lem}

\begin{proof}
The quantity \(\zeta_{(0,M),a}(\boldsymbol{\alpha})\) is a finite linear
combination of truncated multiple zeta values with indices
\[
(\alpha_1+\lambda_1,\cdots,\alpha_d+\lambda_d).
\]
Applying
Theorem~\ref{thm:regularized-Hurwitz-asymptotic} at \(s=0\) to every
such index and summing the resulting finite collection of estimates
gives \eqref{eq:shifted-approximation-TM}.

Since
\[
T_M=\log M+\gamma+O(M^{-1}),
\]
and \(\zeta_{*,a}^X(\boldsymbol{\alpha})\) is a fixed polynomial in \(X\),
the mean value theorem gives
\[
\zeta_{*,a}^{T_M}(\boldsymbol{\alpha})
-
\zeta_{*,a}^{\log M+\gamma}(\boldsymbol{\alpha})
=
O_{\boldsymbol{\alpha},a}
\left(
\frac{(1+\log M)^A}{M}
\right).
\]
This proves \eqref{eq:shifted-approximation-LM}. The polynomial growth
estimate \eqref{eq:shifted-polynomial-growth} is immediate.
\end{proof}

\subsection{Higher-order and simple-pole kernels}

Let $\boldsymbol{\alpha},
\boldsymbol{\beta}$ be two indices, possibly empty, and let \(a,b\geqslant0\). For
\(-N<n<N\), put
\begin{align}
C_N(n)
:=
\zeta_{(-N-n,0),a}(\boldsymbol{\alpha})
\zeta_{(0,N-n),b}(\boldsymbol{\beta}).
\label{eq:def-moving-coefficient}
\end{align}

\begin{lem}[Higher-order-pole approximation]
\label{lem:higher-pole-moving-window}
Let \(k\geqslant2\), and put
\[
d:=\ell(\boldsymbol{\alpha})
+\ell(\boldsymbol{\beta}).
\]
Then, uniformly for $s\in\mathcal{L}, s\ne0$,
one has
\begin{align}
\sum_{-N<n<N}
\frac{C_N(n)}{(n+s)^k}
={}&
C_N(0)
\sum_{-N<n<N}\frac1{(n+s)^k}
\nonumber\\
&+
O_{\boldsymbol{\alpha},\boldsymbol{\beta},a,b,k}
\left(
\frac{(1+\log N)^d}{N}
\right).
\label{eq:higher-pole-moving-window}
\end{align}
\end{lem}

\begin{proof}
The difference between the two sides of
\eqref{eq:higher-pole-moving-window} is
\[
E_N(s)
=
\sum_{0<|n|<N}
\frac{
C_N(n)-C_N(0)
}{
(n+s)^k
}.
\]
Suppose first that \(0<|n|\leqslant N/2\). By
Lemma~\ref{lem:shifted-truncation-estimates},
\[
\left|
C_N(n)- C_N(0)
\right|
\leqslant c_1\cdot
\frac{|n|}{N}
(1+\log N)^{d-1},
\]
for some constant $c_1>0$, with the usual interpretation that the difference is zero when
\(d=0\).

For every nonzero integer \(n\),
\[
|n+s|
\geqslant
|\operatorname{Re}(n+s)|
\geqslant
|n|-\frac12
\geqslant
\frac{|n|}{2}.
\]
Hence
\begin{align*}
\sum_{0<|n|\leqslant N/2}
\frac{
| C_N(n)- C_N(0)|
}{
|n+s|^k
}
&\leqslant c_2\cdot
\frac{(1+\log N)^{d-1}}N
\sum_{0<n\leqslant N/2}\frac1{n^{k-1}}\\
&\leqslant c_2\cdot
\frac{(1+\log N)^d}{N}.
\end{align*}
for some constant $c_2>0$. If \(N/2<|n|<N\), the growth estimate gives
\[
|C_N(n)|+| C_N(0)|
\leqslant c_3\cdot
(1+\log N)^d,
\]
for some $c_3>0$, while
\[
|n+s|\geqslant\frac N4.
\]
Since this range contains \(O(N)\) integers,
\[
\sum_{N/2<|n|<N}
\frac{
|C_N(n)-C_N(0)|
}{
|n+s|^k
}
\leqslant 4^kc_3\cdot
N\frac{(1+\log N)^d}{N^k}
\leqslant 4^kc_3\cdot
\frac{(1+\log N)^d}{N}.
\]
Combining the two ranges proves the lemma.
\end{proof}

\begin{lem}[Simple-pole approximation]
\label{lem:simple-pole-moving-window}
Let $s\in\mathcal{L},s\neq0$. Then, locally uniformly for \(s\in\mathcal{L}\),
\begin{align}
\sum_{-N<n<N}
\frac{ C_N(n)}{n+s}
={}&
 C_N(0)\cdot
\sum_{-N<n<N}\frac1{n+s}
\nonumber\\
&+
\sum_{0<|n|<N}\frac{ C_N(n)}n
+
O_{\boldsymbol{\alpha},\boldsymbol{\beta},a,b,s}
\left(
\frac{(1+\log N)^d}{N}
\right),
\label{eq:simple-pole-moving-window}
\end{align}
where $d=\ell(\boldsymbol{\alpha})
+\ell(\boldsymbol{\beta})$.
\end{lem}

\begin{proof}
Let \(D_N(s)\) denote the left-hand side of
\eqref{eq:simple-pole-moving-window} minus the first two terms on its
right-hand side. Since the \(n=0\) contribution cancels, we have
\begin{align*}
D_N(s)
&=
\sum_{0<|n|<N}
\left(
\frac{ C_N(n)- C_N(0)}{n+s}
-
\frac{ C_N(n)}n
\right)\\
&=
\sum_{0<|n|<N}
\bigl(
 C_N(n)- C_N(0)
\bigr)
\left(
\frac1{n+s}-\frac1n
\right),
\end{align*}
For \(s\) in a fixed compact subset of \(\mathcal{L}\),
\[
\left|
\frac1{n+s}-\frac1n
\right|
=
\frac{|s|}{|n|\,|n+s|}
\leqslant\frac{2|s|}{n^2}.
\]
Therefore
\[
|D_N(s)|
\leqslant2|s|\cdot
\sum_{0<|n|<N}
\frac{
| C_N(n)- C_N(0)|
}{
n^2
}.
\]
By repeating the discussion in Lemma \ref{lem:higher-pole-moving-window}, we immediately obtain the desired result.
\end{proof}

We shall also use the following standard truncation estimate for
monotangent functions.

\begin{lem}[Truncated monotangent kernels]
\label{lem:truncated-monotangent-kernels}
For every compact set $K\subset\mathbb C\setminus\mathbb Z$, one has
\begin{align}
\mathcal{T}_N(k;s)
=
\mathcal T(k;s)
+
\begin{cases}
O_{K}(N^{-1}),&k=1,\\[2mm]
O_{k,K}(N^{1-k}),&k\geqslant2,
\end{cases}
\label{eq:truncated-monotangent-kernels}
\end{align}
uniformly for \(s\in K\).
\end{lem}

\begin{proof}
For \(k\geqslant2\), the result follows directly from the absolute
convergence of the monotangent series. For \(k=1\), pairing the positive and negative tails gives
\[
\mathcal T(1;s)-\mathcal{T}_N(1;s)
=
\sum_{n=N}^{\infty}
\left(
\frac1{n+s}+\frac1{s-n}
\right).
\]
For \(s\in K\), the expression in parentheses is \(O_K(n^{-2})\).
Hence the tail is \(O_K(N^{-1})\).
\end{proof}

\subsection{Moving-window integrals}

\begin{lem}[One-sided moving-window asymptotic]
\label{lem:one-sided-moving-window}
Let $\boldsymbol{\alpha},\boldsymbol{\beta}$ be fixed indices, possibly empty, and let \(a,b\geqslant0\). Then
\begin{align}
&
\sum_{0<n<N}
\frac{
\zeta_{(-N-n,0),a}(\boldsymbol{\alpha})
\zeta_{(0,N-n),b}(\boldsymbol{\beta})
}{n}
\nonumber\\
&=
(-1)^{|\boldsymbol{\alpha}|+a}
\Biggl[
T_N\cdot
\zeta_{*,a}^{T_N}(\overleftarrow{\boldsymbol{\alpha}})
\zeta_{*,b}^{T_N}(\boldsymbol{\beta})
\nonumber\\
&\qquad\qquad+
\int_0^1
\frac{
\zeta_{*,a}^{T_N+\log(1+t)}(\overleftarrow{\boldsymbol{\alpha}})
\zeta_{*,b}^{T_N+\log(1-t)}(\boldsymbol{\beta})
-
\zeta_{*,a}^{T_N}(\overleftarrow{\boldsymbol{\alpha}})
\zeta_{*,b}^{T_N}(\boldsymbol{\beta})
}{t}\,\de t
\Biggr]
+o(1).
\label{eq:one-sided-moving-window}
\end{align}
The integral is understood as an improper integral at \(t=1\).
\end{lem}

\begin{proof}
By \eqref{eq:shifted-negative-reversal}, the left-hand side of
\eqref{eq:one-sided-moving-window} is
\begin{align}
S_N
:=(-1)^{|\boldsymbol{\alpha}|+a}
\sum_{0<n<N}
\frac{
\zeta_{(0,N+n),a}(\overleftarrow{\boldsymbol{\alpha}})
\zeta_{(0,N-n),b}(\boldsymbol{\beta})
}{n}.
\label{eq:def-one-sided-SN}
\end{align}
Using Lemma~\ref{lem:shifted-regularized-approximation}, we first
replace the two finite multiple zeta values in
\eqref{eq:def-one-sided-SN} by
\[
\zeta_{*,a}^{\log(N+n)+\gamma}(\overleftarrow{\boldsymbol{\alpha}}),
\qquad
\zeta_{*,b}^{\log(N-n)+\gamma}(\boldsymbol{\beta}).
\]
To justify this uniformly, take
\[
K_N:=\lfloor N^{1/2}\rfloor
\]
and split the sum into
\[
1\leqslant n\leqslant N-K_N
\]
and
\[
N-K_N<n<N.
\]

On the first range, the total replacement error is bounded by
\[
(1+\log N)^A
\sum_{n=1}^{N-K_N}
\frac1n
\left(
\frac1{N+n}+\frac1{N-n}
\right)
=
o(1).
\]
On the boundary range, there are \(O(N^{1/2})\) terms,
\(1/n=O(N^{-1})\), and every factor has at most polylogarithmic
growth. Hence that range also contributes \(o(1)\). Therefore
\begin{align}
S_N
=
(-1)^{|\boldsymbol{\alpha}|+a}
\sum_{0<n<N}
\frac{
\zeta_{*,a}^{\log(N+n)+\gamma}(\overleftarrow{\boldsymbol{\alpha}})
\zeta_{*,b}^{\log(N-n)+
\gamma}(\boldsymbol{\beta})
}{n}
+o(1).
\label{eq:one-sided-polynomial-replacement}
\end{align}
Since
\[
\log(N\pm n)
=
\log N+\log\left(1\pm\frac nN\right),
\]
define
\[
F_X(t)
:=
\zeta_{*,a}^{X+\log(1+t)}(\overleftarrow{\boldsymbol{\alpha}})
\zeta_{*,b}^{X+\log(1-t)}(\boldsymbol{\beta})
.
\]
Then
\[
S_N
=
(-1)^{|\boldsymbol{\alpha}|+a}
\sum_{0<n<N}
\frac1n
F_{\log N+\gamma}\left(\frac{n}{N}\right)
+o(1).
\]
Since
\[
F_X(0)
=
\zeta_{*,a}^{X}(\overleftarrow{\boldsymbol{\alpha}})
\zeta_{*,b}^{X}(\boldsymbol{\beta}),
\]
we have
\begin{align*}
\sum_{0<n<N}
\frac1nF_{\log N+\gamma}\left(\frac nN\right)
={}&
T_N\cdot
\zeta_{*,a}^{\log N+\gamma}(\overleftarrow{\boldsymbol{\alpha}})
\zeta_{*,b}^{\log N+\gamma}(\boldsymbol{\beta})\\
&+
\frac1N
\sum_{n=1}^{N-1}
G_{\log N+\gamma}\left(\frac nN\right),
\end{align*}
where
\[
G_X(t):=\frac{F_X(t)-F_X(0)}t.
\]

The singularity of \(G_X(t)\) at \(t=0\) is removable. Near \(t=1\),
the function has at most polynomial growth in
\(|\log(1-t)|\). Since \(\zeta_{*,a}^{X}(\overleftarrow{\boldsymbol{\alpha}})\) and \(
\zeta_{*,b}^{X}(\boldsymbol{\beta})\) are fixed polynomials, there are
constants \(C,D>0\) such that
\[
|G_X(t)|
\leqslant
C\bigl(1+|X|+|\log(1-t)|\bigr)^D
\]
and
\[
|G_X'(t)|
\leqslant
\frac{
C\bigl(1+|X|+|\log(1-t)|\bigr)^D
}{1-t}.
\]
It follows that
\[
\frac1N
\sum_{n=1}^{N-1}
G_{\log N+\gamma}\left(\frac nN\right)
=
\int_0^1G_{\log N+\gamma}(t)\,\de t+o(1).
\]

Finally,
\[
T_N=\log N+\gamma+O(N^{-1}).
\]
Since all functions of \(\log N+\gamma\) occurring above are polynomials in
\(\log N+\gamma\), replacing \(\log N+\gamma\) by \(T_N\) changes the result only by
\(o(1)\). Substitution into \eqref{eq:def-one-sided-SN} gives
\eqref{eq:one-sided-moving-window}.
\end{proof}

The symmetric sum needed for the simple-pole contribution is the
following direct consequence.

\begin{cor}[Symmetric moving-window asymptotic]
\label{cor:symmetric-moving-window}
Under the assumptions of
Lemma~\ref{lem:one-sided-moving-window}, one has
\begin{align}
&
\sum_{0<|n|<N}
\frac{
\zeta_{(-N-n,0),a}(\boldsymbol{\alpha})
\zeta_{(0,N-n),b}(\boldsymbol{\beta})
}{n}
\nonumber\\
&=
(-1)^{|\boldsymbol{\alpha}|+a}
\int_0^1\frac1t
\Bigl[
\zeta_{*,a}^{T_N+\log(1+t)}(\overleftarrow{\boldsymbol{\alpha}})
\zeta_{*,b}^{T_N+\log(1-t)}(\boldsymbol{\beta})
\nonumber\\
&\hspace{31mm}
-
\zeta_{*,a}^{T_N+\log(1-t)}(\overleftarrow{\boldsymbol{\alpha}})
\zeta_{*,b}^{T_N+\log(1+t)}(\boldsymbol{\beta})
\Bigr]\,\de t
+o(1).
\label{eq:symmetric-moving-window}
\end{align}
\end{cor}

\begin{proof}
Pairing the summands corresponding to \(n\) and \(-n\), and using
\eqref{eq:shifted-negative-reversal}, gives
\begin{align*}
&
\sum_{0<|n|<N}
\frac{
\zeta_{(-N-n,0),a}(\boldsymbol{\alpha})
\zeta_{(0,N-n),b}(\boldsymbol{\beta})
}{n}
\\
&=
(-1)^{|\boldsymbol{\alpha}|+a}
\sum_{0<n<N}\frac1n
\Bigl[
\zeta_{(0,N+n),a}(\overleftarrow{\boldsymbol{\alpha}})
\zeta_{(0,N-n),b}(\boldsymbol{\beta})
\\
&\hspace{43mm}
-
\zeta_{(0,N-n),a}(\overleftarrow{\boldsymbol{\alpha}})
\zeta_{(0,N+n),b}(\boldsymbol{\beta})
\Bigr].
\end{align*}
Applying the polynomial replacement and Riemann-sum argument from the
proof of Lemma~\ref{lem:one-sided-moving-window} to this difference
gives \eqref{eq:symmetric-moving-window}. The two harmonic terms
containing \(T_N\) cancel.
\end{proof}

\subsection{Asymptotic reduction of the symmetric truncation}

For \(1\leqslant j\leqslant r\) and \(0\leqslant m\leqslant k_j-1\), define
\begin{align}
C_{j,m}^{X}(\bfk)
:=
\sum_{\substack{a+b=m\\a,b\geqslant0}}
(-1)^{|\bfk_{[1,j)}|+a}
\zeta_{*,a}^{X}
\bigl(\overleftarrow{\bfk_{[1,j)}}\bigr)
\zeta_{*,b}^{X}
\bigl(\bfk_{(j,r]}\bigr).
\label{eq:def-CjmX}
\end{align}
Define
\begin{align}
\mathcal R_{\bfk}(X,Y)
:=
\sum_{j=1}^{r}
\sum_{\substack{a+b=k_j-1\\a,b\geqslant0}}
(-1)^{|\bfk_{[1,j)}|+a}
\zeta_{*,a}^{X}
\bigl(\overleftarrow{\bfk_{[1,j)}}\bigr)
\zeta_{*,b}^{Y}
\bigl(\bfk_{(j,r]}\bigr).
\label{eq:def-residue-polynomial}
\end{align}
Then put
\begin{align}
C_1^{X}(\bfk)
:=
\mathcal R_{\bfk}(X,X)
=
\sum_{j=1}^{r}
C_{j,k_j-1}^{X}(\bfk),
\label{eq:def-C1X}
\end{align}
and
\begin{align}
C_0^{X}(\bfk)
:=
\int_0^1
\frac{
\mathcal R_{\bfk}
(X+\log(1+t),X+\log(1-t))
-
\mathcal R_{\bfk}
(X+\log(1-t),X+\log(1+t))
}{t}\,\de t.
\label{eq:def-C0X}
\end{align}
The integral converges because the numerator vanishes at \(t=0\), and
has at most polynomial growth in \(\log(1-t)\) as \(t\to1^{-}\).
\begin{re}
    The quantity $C_0^{X}(\bfk)$ can be regarded as a polynomial in $X$ whose coefficients are alternating MZVs. Indeed, we will eventually prove that 
    $$C_0^{X}(\bfk)=\begin{cases}
        \frac{(i\pi)^r}{r!}, \bfk=\{1\}^r,2\mid r,\\
        0,\mathrm{otherwise}.
    \end{cases}$$
    see Section 6.
\end{re}

\begin{thm}[Asymptotic reduction]
\label{thm:asymptotic-reduction}
Let $\bfk=(k_1,\cdots,k_r)$ be an index, and let $K\subset\mathcal{L}$ be compact. Then, uniformly for \(s\in K\),
\begin{align}
\mathcal T_N(\bfk;s)
={}&
C_0^{T_N}(\bfk)
+
C_1^{T_N}(\bfk)\mathcal T(1;s)
\nonumber\\
&+
\sum_{j=1}^{r}
\sum_{m=0}^{k_j-2}
C_{j,m}^{T_N}(\bfk)
\mathcal T(k_j-m;s)
+
o_{\bfk,K}(1).
\label{eq:asymptotic-reduction}
\end{align}
If \(k_j=1\), the corresponding inner sum is empty.
\end{thm}

\begin{proof}
Starting from the finite reduction formula of
Theorem~\ref{thm:finite_pf_reduction}, for fixed \(j,a,b\), put
\[
A_{j,a,N}(n)
:=
\zeta_{(-N-n,0),a}(\bfk_{[1,j)}),
\qquad
B_{j,b,N}(n)
:=
\zeta_{(0,N-n),b}(\bfk_{(j,r]}).
\]
We first consider
\[
0\leqslant m\leqslant k_j-2,
\qquad
a+b=m.
\]
Then
\[
q:=k_j-m\geqslant2.
\]
By Lemma~\ref{lem:higher-pole-moving-window},
\begin{align}
\sum_{-N<n<N}
\frac{
A_{j,a,N}(n)B_{j,b,N}(n)
}{
(n+s)^q
}
={}&
A_{j,a,N}(0)B_{j,b,N}(0)
\mathcal{T}_N(q;s)
+o_{\bfk,K}(1).
\label{eq:higher-pole-in-reduction}
\end{align}
By
\eqref{eq:shifted-negative-reversal} and
Lemma~\ref{lem:shifted-regularized-approximation},
\begin{align*}
A_{j,a,N}(0)
&=
(-1)^{|\boldsymbol{\alpha}_j|+a}
\zeta_{*,a}^{T_N}
(\overleftarrow{\boldsymbol{\alpha}_j})
+
O\left(\frac{(1+\log N)^A}{N}\right),\\
B_{j,b,N}(0)
&=
\zeta_{*,b}^{T_N}
(\boldsymbol{\beta}_j)
+
O\left(\frac{(1+\log N)^A}{N}\right).
\end{align*}
Using Lemma~\ref{lem:truncated-monotangent-kernels}, we therefore get
\begin{align}
\sum_{-N<n<N}
\frac{
A_{j,a,N}(n)B_{j,b,N}(n)
}{
(n+s)^q
}
={}&
(-1)^{|\boldsymbol{\alpha}_j|+a}
\zeta_{*,a}^{T_N}
(\overleftarrow{\boldsymbol{\alpha}_j})
\zeta_{*,b}^{T_N}
(\boldsymbol{\beta}_j)
\mathcal T(q;s)
\nonumber\\
&+
o_{\bfk,K}(1).
\label{eq:higher-pole-final-reduction}
\end{align}
Summing over \(a+b=m\) gives
\[
C_{j,m}^{T_N}(\bfk)
\mathcal T(k_j-m;s)
+
o_{\bfk,K}(1).
\]
It remains to consider the simple-pole case $m=k_j-1$. By Lemma~\ref{lem:simple-pole-moving-window},
\begin{align}
&
\sum_{-N<n<N}
\frac{
A_{j,a,N}(n)B_{j,b,N}(n)
}{n+s}
\nonumber\\
&=
A_{j,a,N}(0)B_{j,b,N}(0)
\mathcal{T}_N(1;s)
+
\sum_{0<|n|<N}
\frac{
A_{j,a,N}(n)B_{j,b,N}(n)
}{n}
+
o_{\bfk,K}(1).
\label{eq:simple-pole-reduction-step}
\end{align}
The first term on the right-hand side satisfies
\begin{align}
A_{j,a,N}(0)B_{j,b,N}(0)
\mathcal{T}_N(1;s)
\nonumber=
(-1)^{|\boldsymbol{\alpha}_j|+a}
\zeta_{*,a}^{T_N}
(\overleftarrow{\boldsymbol{\alpha}_j})
\zeta_{*,b}^{T_N}
(\boldsymbol{\beta}_j)
\mathcal T(1;s)
+
o_{\bfk,K}(1).
\label{eq:simple-pole-monotangent-part}
\end{align}
After summing over \(j\) and
\(a+b=k_j-1\), this gives
\[
C_1^{T_N}(\bfk)\mathcal T(1;s)
+
o_{\bfk,K}(1).
\]
For the second term in
\eqref{eq:simple-pole-reduction-step}, Corollary
\ref{cor:symmetric-moving-window} gives
\begin{align*}
&
\sum_{0<|n|<N}
\frac{
A_{j,a,N}(n)B_{j,b,N}(n)
}{n}
\\
&=
(-1)^{|\boldsymbol{\alpha}_j|+a}
\int_0^1\frac1t
\Bigl[
\zeta_{*,a}^{T_N+\log(1+t)}
(\overleftarrow{\boldsymbol{\alpha}_j})
\zeta_{*,b}^{T_N+\log(1-t)}
(\boldsymbol{\beta}_j)
\\
&\hspace{28mm}
-
\zeta_{*,a}^{T_N+\log(1-t)}
(\overleftarrow{\boldsymbol{\alpha}_j})
\zeta_{*,b}^{T_N+\log(1+t)}
(\boldsymbol{\beta}_j)
\Bigr]\,\de t
+o(1).
\end{align*}
Summing this identity over \(j\) and
\(a+b=k_j-1\) produces exactly
\[
C_0^{T_N}(\bfk)+o(1).
\]
Combining the higher-order-pole and simple-pole contributions with
Theorem~\ref{thm:finite_pf_reduction} proves
\eqref{eq:asymptotic-reduction}.
\end{proof}

\begin{re}
At this stage, the coefficients in
\eqref{eq:asymptotic-reduction} are polynomials evaluated at the
growing parameter \(T_N\). In the next section, this asymptotic
identity will be combined with
\[
\mathcal T_N(\bfk;s)
=
\mathcal T_*^{T_N}(\bfk;s)+o(1)
\]
and a polynomial-identity argument to obtain the exact reduction
formula for an arbitrary regularization parameter \(T\).
\end{re}

\section{The Reduction Theorem and Regularized MZV Relations}

In this section, we complete the proof of the reduction theorem.
We first evaluate the regularized multitangent functions associated
with the exceptional pure-one indices
\[
\{1\}^r=(\underbrace{1,\cdots,1}_{r})
\]
by means of a sine-quotient generating function. We then establish the
decay of all remaining regularized multitangent functions along the
positive and negative imaginary directions.

Combining the asymptotic expansion from Section~3 with the
asymptotic reduction formula from Section~5, and using the fact that
all relevant expressions are polynomials in the regularization
parameter, we obtain an exact reduction formula for an arbitrary
parameter \(T\). The imaginary-axis limits determine the constant term
and the coefficient of the simple monotangent
\(\mathcal T(1;s)\). Finally, comparison with the local Laurent
expansion gives a family of relations among stuffle-regularized
multiple zeta values.

\subsection{The exceptional pure-one indices}

\begin{lem}[Pure-one generating function]
\label{lem:pure-one-generating-function}
For every \(s\in\mathbb C-\mathbb Z\) and every regularization
parameter \(T\), one has
\begin{align}
\sum_{r=0}^{\infty}
\mathcal T_*^T(\{1\}^r;s)u^r
=
\frac{\sin\pi(s+u)}{\sin\pi s}.
\label{eq:pure-one-generating-function}
\end{align}
Consequently,
\begin{align}
\boxed{\mathcal T_*^T(\{1\}^r;s)
=
\begin{cases}
\displaystyle
(-1)^m\frac{\pi^{2m}}{(2m)!},
&
r=2m,
\\[10pt]
\displaystyle
(-1)^m
\frac{\pi^{2m+1}}{(2m+1)!}
\cot(\pi s),
&
r=2m+1.
\end{cases}
\label{eq:pure-one-explicit}}
\end{align}
In particular, these functions are independent of \(T\).
\end{lem}

\begin{proof}

Consider the generating function
\[
F_N(u;s)
:=
\sum_{r=0}^{\infty}
\mathcal{T}_N(\{1\}^r;s)u^r .
\]
Although $F_N(u;s)$ is introduced as a power series, for fixed $N$ it is in fact a polynomial of degree at most $2N-1$, since there are only $2N-1$ available integers in the truncation interval. By the elementary symmetric polynomial identity, we have

\begin{align}
F_N(u;s)=\sum_{r=0}^{2N-1}\mathcal{T}_N(\{1\}^r;s)u^r=\prod_{-N<n<N}
\frac{n+s+u}{n+s}=
\frac{
\Gamma(N+s+u)
\Gamma(1-N+s)
}
{
\Gamma(N+s)
\Gamma(1-N+s+u)
}.
\label{eq:gamma-product}
\end{align}
Using the reflection formula
\[
\Gamma(z)\Gamma(1-z)
=
\frac{\pi}{\sin\pi z},
\]
we rewrite the second ratio in \eqref{eq:gamma-product}. We obtain
\begin{align}
F_N(u;s)
&=
\frac{\sin\pi(s+u)}
{\sin\pi s}
\frac{
\Gamma(N+s+u)
\Gamma(N-s-u)
}
{
\Gamma(N+s)
\Gamma(N-s)
}.
\label{eq:reflection-form}
\end{align}
Now we use the standard asymptotic formula
\[
\frac{\Gamma(N+a)}{\Gamma(N+b)}
=
N^{a-b}(1+O(N^{-1}))
\]
as \(N\to\infty\). Therefore,
\[
\frac{
\Gamma(N+s+u)
\Gamma(N-s-u)
}
{
\Gamma(N+s)
\Gamma(N-s)
}
=
N^uN^{-u}(1+O(N^{-1}))
=
1+O(N^{-1}).
\]
Consequently,
\[
\lim_{N\to\infty}F_N(u;s)
=
\frac{\sin\pi(s+u)}
{\sin\pi s}.
\]
This proves the generating function identity
\[
\sum_{r=0}^{\infty}
\left(
\lim_{N\to\infty}\mathcal{T}_N(1^r;s)
\right)u^r
=
\frac{\sin\pi(s+u)}
{\sin\pi s}=\cos(\pi u)
+
\cot(\pi s)\sin(\pi u).
\]
This completes the proof.
\end{proof}

\subsection{Vertical limits}

We first record a growth estimate which complements
Corollary~\ref{cor:regularized-Hurwitz-vertical-decay}.

\begin{lem}[Polylogarithmic growth]
\label{lem:regularized-Hurwitz-polylog-growth}
Let \(\bfk\) be any index, possibly empty, and let
\(T\in\mathbb C\) be fixed. Then there exists
\(A_{\bfk}\geqslant0\) such that
\begin{align}
\zeta_*^{T-H(\pm iy)}
(\bfk;\pm iy)
=
O_{\bfk}
\left(
(1+\log|y|)^{A_{\bfk}}
\right)
\label{eq:regularized-Hurwitz-polylog-growth}
\end{align}
as \(|y|\to+\infty\). If \(\bfk\) contains at least one component greater
than \(1\), then one has the stronger estimate
\begin{align}
\zeta_*^{T-H(\pm iy)}
(\bfk;\pm iy)
=
O_{\bfk}
\left(
\frac{(1+\log|y|)^{A_{\bfk}}}{|y|}
\right).
\label{eq:regularized-Hurwitz-power-decay}
\end{align}
\end{lem}

\begin{proof}
By the polynomial form of the stuffle regularization,
\[
\zeta_*^X(\bfk;s)
=
\sum_{l,\boldsymbol{\alpha}}
c_{l,\boldsymbol{\alpha}}\,
X^l\zeta(\boldsymbol{\alpha};s),
\]
where every \(\boldsymbol{\alpha}\) is admissible or empty. Since
\[
T-H(\pm iy)=O(1+\log|y|),
\]
and every nonempty admissible multiple Hurwitz zeta function has at
most polynomial logarithmic growth along the imaginary axis,
\eqref{eq:regularized-Hurwitz-polylog-growth} follows.

If \(\bfk\) contains a component greater than \(1\),
the empty index does not occur in its stuffle decomposition, as shown
in the proof of
Corollary~\ref{cor:regularized-Hurwitz-vertical-decay}. Applying that
corollary to each nonempty admissible term gives
\eqref{eq:regularized-Hurwitz-power-decay}.
\end{proof}

\begin{lem}[Vertical decay of regularized multitangent functions]
\label{lem:regularized-multitangent-vertical-decay}
Let $\bfk=(k_1,\cdots,k_r)$ be an index. Assume that \(k_j\geqslant2\) for at least one
\(j\in\{1,\cdots,r\}\). Then, we have
\begin{align}
\lim_{y\to\pm\infty}
\mathcal T_*^T(\bfk;iy)=0.
\label{eq:regularized-multitangent-vertical-decay}
\end{align}
\end{lem}

\begin{proof}
We treat \(y\to+\infty\); the proof for \(y\to-\infty\) is identical.

Consider first a term in the first sum of
Definition~\ref{def:regularized-multitangent}:
\begin{align}
&
\zeta_*^{T-H(-iy)}
\bigl(\overleftarrow{\bfk_{[1,l]}};-iy\bigr)
\zeta_*^{T-H(iy)}
\bigl(\bfk_{(l,r]};iy\bigr),
\qquad
0\leqslant l\leqslant r.
\label{eq:first-trifactor-term}
\end{align}

If \(l\geqslant j\), the left index $\overleftarrow{\bfk_{[1,l]}}$ contains \(k_j>1\). By
\eqref{eq:regularized-Hurwitz-power-decay}, the left factor is
\[
O\left(
|y|^{-1}(1+\log|y|)^A
\right),
\]
while the right factor is
\[
O\left(
(1+\log|y|)^B
\right)
\]
by \eqref{eq:regularized-Hurwitz-polylog-growth}. Hence the product
tends to zero.

If \(l<j\), then \(k_j\geqslant2\) belongs to the right index
\(\bfk_{(l,r]}\), and the same argument, with the two factors
interchanged, shows that \eqref{eq:first-trifactor-term} tends to zero.
Thus every term in the first sum tends to zero.

A term in the second sum has the form
\begin{align}
\frac{(-1)^{|\bfk_{[1,l)}|}}{(iy)^{k_l}}
\zeta_*^{T-H(-iy)}
\bigl(\overleftarrow{\bfk_{[1,l)}};-iy\bigr)
\zeta_*^{T-H(iy)}
\bigl(\bfk_{(l,r]};iy\bigr).
\label{eq:second-trifactor-term}
\end{align}
Both Hurwitz factors have at most polylogarithmic growth. Therefore,
for some \(A\geqslant0\),
\[
\eqref{eq:second-trifactor-term}
=
O\left(
\frac{(1+\log|y|)^A}{|y|^{k_l}}
\right)
=o(1),
\]
because \(k_l\geqslant1\).

There are only finitely many terms in the definition of
\(\mathcal T_*^T(\bfk;iy)\), so the result follows.
\end{proof}

\subsection{The exact reduction theorem}

Recall from Section~5 that, for
\(0\leqslant m\leqslant k_j-1\),
\begin{align}
C_{j,m}^{T}(\bfk)
:=
\sum_{\substack{a+b=m\\a,b\geqslant0}}
(-1)^{|\bfk_{[1,j)}|+a}
\zeta_{*,a}^{T}
\bigl(\overleftarrow{\bfk_{[1,j)}}\bigr)
\zeta_{*,b}^{T}
\bigl(\bfk_{(j,r]}\bigr).
\label{eq:recall-CjmT}
\end{align}
Define
\begin{align}
c_0(\bfk)
:=
\begin{cases}
\displaystyle
\frac{(i\pi)^r}{r!},
&
\bfk=\{1\}^r
, 2\mid r,
\\[8pt]
0,
&
\text{otherwise},
\end{cases}
\label{eq:def-delta-k}
\end{align}
and
\begin{align}
c_1(\bfk)
:=
\begin{cases}
\displaystyle
(-1)^{(r-1)/2}
\frac{\pi^{r-1}}{r!},
&
\bfk=\{1\}^r
, 2\nmid r,
\\[8pt]
0,
&
\text{otherwise}.
\end{cases}
\label{eq:def-rho-k}
\end{align}

\begin{thm}[Regularized reduction theorem]
\label{thm:regularized-reduction}
Let $\bfk=(k_1,\cdots,k_r)$ be an index. For every regularization parameter \(T\) and every
\(s\in\mathbb C-\mathbb Z\), one has
\begin{align}
\mathcal T_*^T(\bfk;s)
=
c_0(\bfk)
+
c_1(\bfk)\mathcal T(1;s)
+
\sum_{j=1}^{r}
\sum_{m=0}^{k_j-2}
C_{j,m}^{T}(\bfk)
\mathcal T(k_j-m;s).
\label{eq:regularized-reduction}
\end{align}
When \(k_j=1\), the corresponding inner sum is empty.
\end{thm}

\begin{proof}
We first prove an exact reduction formula containing the coefficients
\(C_0^T(\bfk)\) and \(C_1^T(\bfk)\) defined in Section~5.

Fix $s\in\mathcal{L},s\neq0$. For an indeterminate \(X\), define
\begin{align}
 D_{\bfk,s}(X)
:={}&
\mathcal T_*^X(\bfk;s)
-
C_0^X(\bfk)
-
C_1^X(\bfk)\mathcal T(1;s)
\nonumber\\
&-
\sum_{j=1}^{r}
\sum_{m=0}^{k_j-2}
C_{j,m}^{X}(\bfk)
\mathcal T(k_j-m;s).
\label{eq:def-reduction-difference-polynomial}
\end{align}
Every term on the right-hand side is a polynomial in \(X\).
In particular, \(C_0^X(\bfk)\) is a polynomial because its defining
integrand is polynomial in \(X\), and the integral may be taken
coefficientwise.

By Theorem~\ref{thm:symmetric-truncation-asymptotic},
\[
\mathcal T_N(\bfk;s)
=
\mathcal T_*^{T_N}(\bfk;s)+o(1),
\]
whereas Theorem~\ref{thm:asymptotic-reduction} gives
\[
\begin{aligned}
\mathcal T_N(\bfk;s)
=
C_0^{T_N}(\bfk)
+
C_1^{T_N}(\bfk)\cdot\mathcal T(1;s)+
\sum_{j=1}^{r}
\sum_{m=0}^{k_j-2}
C_{j,m}^{T_N}(\bfk)
\mathcal T(k_j-m;s)
+o(1).
\end{aligned}
\]
Subtracting these two asymptotic identities yields
\[
D_{\bfk,s}(T_N)=o(1).
\]
Since \(T_N\to+\infty\) and
\( D_{\bfk,s}(X)\) is a polynomial, it follows that
\[
 D_{\bfk,s}(X)\equiv0.
\]
Therefore,
\begin{align}
\mathcal T_*^T(\bfk;s)
=
C_0^T(\bfk)
+
C_1^T(\bfk)\cdot\mathcal T(1;s)
+
\sum_{j=1}^{r}
\sum_{m=0}^{k_j-2}
C_{j,m}^{T}(\bfk)
\mathcal T(k_j-m;s)
\label{eq:preliminary-exact-reduction}
\end{align}
on the above punctured strip.

Both sides of \eqref{eq:preliminary-exact-reduction} are meromorphic
and \(1\)-periodic in \(s\). Since they agree on a nonempty open set,
the identity theorem extends
\eqref{eq:preliminary-exact-reduction} to all $s\in\mathbb C-\mathbb Z.$

It remains to determine \(C_0^T(\bfk)\) and \(C_1^T(\bfk)\).
For \(k\geqslant2\),
\[
\mathcal T(k;iy)
=
O_k(|y|^{1-k}),
\]
and hence
\[
\lim_{y\to\pm\infty}\mathcal T(k;iy)=0.
\]
Moreover,
\[
\lim_{y\to+\infty}\mathcal T(1;iy)=-i\pi,
\qquad
\lim_{y\to-\infty}\mathcal T(1;iy)=i\pi.
\]
Suppose first that \(\bfk\ne\{1\}^r\). Equivalently, some component of
\(\bfk\) is greater than \(1\). By
Lemma~\ref{lem:regularized-multitangent-vertical-decay},
\[
\lim_{y\to\pm\infty}\mathcal T_*^T(\bfk;iy)=0.
\]
Taking \(y\to+\infty\) and \(y\to-\infty\) in
\eqref{eq:preliminary-exact-reduction}, respectively, gives
\[
C_0^T(\bfk)-i\pi C_1^T(\bfk)=0,
\]
and
\[
C_0^T(\bfk)+i\pi C_1^T(\bfk)=0.
\]
Thus
\[
C_0^T(\bfk)=C_1^T(\bfk)=0.
\]
Now let \(\bfk=\{1\}^r\). In this case, the double sum in
\eqref{eq:preliminary-exact-reduction} is empty, and
Lemma~\ref{lem:pure-one-generating-function} gives
\[
\mathcal T_*^T(\{1\}^r;s)
=
\begin{cases}
\displaystyle
\frac{(i\pi)^r}{r!},
&
2\mid r,
\\[8pt]
\displaystyle

\frac{(i\pi)^{r-1}}{r!}\,
\mathcal T(1;s),
&
2\nmid r.
\end{cases}
\]
It follows that
\[
C_0^T(\{1\}^r)=c_0(\{1\}^r),
\qquad
C_1^T(\{1\}^r)=c_1(\{1\}^r).
\]
Therefore, for every index \(\bfk\),
\[
C_0^T(\bfk)=c_0(\bfk),
\qquad
C_1^T(\bfk)=c_1(\bfk).
\]
Substitution into
\eqref{eq:preliminary-exact-reduction} proves
\eqref{eq:regularized-reduction}.
\end{proof}

\begin{re}
Equivalently, the two exceptional coefficients are characterized by
the cusp identities
\begin{align}
C_0^T(\bfk)-i\pi C_1^T(\bfk)
&=
\begin{cases}
\dfrac{(-i\pi)^r}{r!},
&\bfk=\{1\}^r,\\[6pt]
0,&\bfk\ne\{1\}^r,
\end{cases}
\label{eq:upper-cusp-coefficients}
\\
C_0^T(\bfk)+i\pi C_1^T(\bfk)
&=
\begin{cases}
\dfrac{(i\pi)^r}{r!},
&\bfk=\{1\}^r,\\[6pt]
0,&\bfk\ne\{1\}^r.
\end{cases}
\label{eq:lower-cusp-coefficients}
\end{align}
\end{re}

\begin{re}
With the normalization \(T=0\), Theorem
\ref{thm:regularized-reduction} recovers, up to the index and
notation conventions used here, Bouillot's regularized reduction
formula \cite{Bouillot2014}. For instance, $\mathcal T_*^T(2,1;s)=T\cdot\mathcal T(2;s)$, and hence $\mathcal T_*^0(2,1;s)=0$,
which agrees with Bouillot's convention for this index.
\end{re}

\begin{cor}[Convergent reduction]
\label{cor:convergent-reduction}
If $k_1,k_r\geqslant2$, then
\[
\mathcal T(\bfk;s)
=
\sum_{j=1}^{r}
\sum_{m=0}^{k_j-2}
C_{j,m}^0(\bfk)\,
\mathcal T(k_j-m;s),
\]
where
\[
C_{j,m}^0(\bfk)
=
\sum_{\substack{a+b=m\\a,b\geqslant0}}
(-1)^{|\bfk_{[1,j)}|+a}
\zeta_{*,a}^0
\bigl(\overleftarrow{\bfk_{[1,j)}}\bigr)
\zeta_{*,b}^0
\bigl(\bfk_{(j,r]}\bigr).
\]
All multiple zeta values occurring here are convergent, and the
coefficients are independent of \(T\).
\end{cor}

\begin{proof}
By Corollary~\ref{cor:compatibility-convergent-multitangent},
\[
\mathcal T_*^T(\bfk;s)=\mathcal T(\bfk;s).
\]
Moreover, since \(k_1,k_r\geqslant2\), the reversed prefix
\(\overleftarrow{\bfk_{[1,j)}}\) and the suffix
\(\bfk_{(j,r]}\) are admissible whenever they are nonempty. Hence the
regularized values appearing in the coefficients are ordinary
convergent MZVs.
\end{proof}

\subsection{Relations among stuffle-regularized MZVs}

For \(k\geqslant1\) and \(m\geqslant0\), define
\begin{align}
\mu_m(k)
:=
\begin{cases}
0,
&(k,m)=(1,0),
\\[4pt]
\displaystyle
\bigl(1+(-1)^{k+m}\bigr)
\binom{-k}{m}\zeta(k+m),
&k+m>1.
\end{cases}
\label{eq:def-mu-mq}
\end{align}

\begin{lem}[Local expansion of a monotangent]
\label{lem:local-monotangent-expansion}
For every \(k\geqslant1\) and \(0<|s|<1\),
\begin{align}
\mathcal T(k;s)
=
\frac1{s^k}
+
\sum_{m=0}^{\infty}
\mu_m(k)s^m.
\label{eq:local-monotangent-expansion}
\end{align}
\end{lem}

\begin{proof}
Separating the term \(n=0\) and pairing the terms \(n\) and \(-n\),
we have
\[
\mathcal T(k;s)
=
\frac1{s^k}
+
\sum_{n=1}^{\infty}
\left(
\frac1{(n+s)^k}
+
\frac1{(s-n)^k}
\right).
\]
For \(|s|<1\), the binomial expansions give
\[
\frac1{(n+s)^k}
=
\sum_{m=0}^{\infty}
\binom{-k}{m}
\frac{s^m}{n^{k+m}},
\]
and
\[
\frac1{(s-n)^k}
=
\sum_{m=0}^{\infty}
(-1)^{k+m}
\binom{-k}{m}
\frac{s^m}{n^{k+m}}.
\]
Adding the two series and summing over \(n\geqslant1\) yields
\eqref{eq:local-monotangent-expansion}. When \(k=1,m=0\), the paired
constant term is zero, which explains the convention
\(\mu_0(1)=0\).
\end{proof}

Recall the coefficients \(A_m^T(\bfk)\) and \(B_{j,m}^T(\bfk)\) from
Section~3:
\begin{align}
A_m^T(\bfk)
={}&
\sum_{j=0}^{r}
(-1)^{|\bfk_{[1,j]}|}
\sum_{\substack{a+b=m\\a,b\geqslant0}}
(-1)^a
\zeta_{*,a}^T
\bigl(\overleftarrow{\bfk_{[1,j]}}\bigr)
\zeta_{*,b}^T
\bigl(\bfk_{(j,r]}\bigr),
\label{eq:recall-AmT}
\\
B_{j,m}^T(\bfk)
={}&
(-1)^{|\bfk_{[1,j)}|}
\sum_{\substack{a+b=m\\a,b\geqslant0}}
(-1)^a
\zeta_{*,a}^T
\bigl(\overleftarrow{\bfk_{[1,j)}}\bigr)
\zeta_{*,b}^T
\bigl(\bfk_{(j,r]}\bigr).
\label{eq:recall-BjmT}
\end{align}
For \(m\geqslant0\), define
\begin{align}
\Theta_{*,m}^T(\bfk)
:=
A_m^T(\bfk)
+
\sum_{j=1}^{r}
B_{j,k_j+m}^T(\bfk)
-
\sum_{j=1}^{r}
\sum_{h=0}^{k_j-1}
B_{j,h}^T(\bfk)\,
\mu_m(k_j-h).
\label{eq:def-Theta}
\end{align}
Equivalently,
\begin{align}
\Theta_{*,m}^T(\bfk)
={}&
\sum_{j=0}^{r}
(-1)^{|\bfk_{[1,j]}|}
\sum_{\substack{a+b=m\\a,b\geqslant0}}
(-1)^a
\zeta_{*,a}^T
\bigl(\overleftarrow{\bfk_{[1,j]}}\bigr)
\zeta_{*,b}^T
\bigl(\bfk_{(j,r]}\bigr)
\nonumber\\
&+
\sum_{j=1}^{r}
(-1)^{|\bfk_{[1,j)}|}
\sum_{\substack{a+b=k_j+m\\a,b\geqslant0}}
(-1)^a
\zeta_{*,a}^T
\bigl(\overleftarrow{\bfk_{[1,j)}}\bigr)
\zeta_{*,b}^T
\bigl(\bfk_{(j,r]}\bigr)
\nonumber\\
&-
\sum_{j=1}^{r}
\sum_{h=0}^{k_j-1}
(-1)^{|\bfk_{[1,j)}|}
\sum_{\substack{a+b=h\\a,b\geqslant0}}
(-1)^a
\zeta_{*,a}^T
\bigl(\overleftarrow{\bfk_{[1,j)}}\bigr)
\zeta_{*,b}^T
\bigl(\bfk_{(j,r]}\bigr)
\mu_m(k_j-h).
\label{eq:Theta-expanded}
\end{align}

\begin{thm}[Relations among regularized MZVs]
\label{thm:regularized-MZV-relations}
For every index \(\bfk\) and every \(m\geqslant0\),
\begin{align}
\Theta_{*,m}^T(\bfk)
=\delta_{m,0}\cdot c_0(\bfk).
\label{eq:regularized-MZV-relations}
\end{align}
Equivalently,
\[
\Theta_{*,m}^T(\bfk)
=
\begin{cases}
\displaystyle
\frac{(i\pi)^r}{r!},
&
\bfk=\{1\}^r,m=0,2\mid r,
\\[8pt]
0,
&
\mathrm{otherwise}.
\end{cases}
\]
\end{thm}

\begin{proof}
By Lemma~\ref{lem:local-multitangent-expansion}, the coefficient of
\(s^m\) in the local Laurent expansion of
\(\mathcal T_*^T(\bfk;s)\) is
\begin{align}
A_m^T(\bfk)
+
\sum_{j=1}^{r}B_{j,k_j+m}^T(\bfk).
\label{eq:left-local-coefficient}
\end{align}
On the other hand, Theorem~\ref{thm:regularized-reduction} may be
written in the compact form
\begin{align}
\mathcal T_*^T(\bfk;s)
=
c_0(\bfk)
+
\sum_{j=1}^{r}
\sum_{h=0}^{k_j-1}
B_{j,h}^T(\bfk)\,
\mathcal T(k_j-h;s).
\label{eq:compact-reduction}
\end{align}
Indeed,
\[
\sum_{j=1}^{r}B_{j,k_j-1}^T(\bfk)
=
C_1^T(\bfk)
=c_1(\bfk).
\]
Using Lemma~\ref{lem:local-monotangent-expansion}, the coefficient of
\(s^m\) on the right-hand side of
\eqref{eq:compact-reduction} is
\begin{align}
\delta_{m,0}\cdot c_0(\bfk)
+
\sum_{j=1}^{r}
\sum_{h=0}^{k_j-1}
B_{j,h}^T(\bfk)\,
\mu_m(k_j-h).
\label{eq:right-local-coefficient}
\end{align}
Comparing \eqref{eq:left-local-coefficient} and
\eqref{eq:right-local-coefficient} gives
\[
\Theta_{*,m}^T(\bfk)
=
\delta_{m,0}\cdot c_0(\bfk),
\]
as required.
\end{proof}

\begin{cor}[The constant-term relation]
\label{cor:constant-term-relation}
For every index \(\bfk\),
\begin{align}
\Theta_{*,0}^T(\bfk)
=
\begin{cases}
\displaystyle
\frac{(i\pi)^r}{r!},
&
\bfk=\{1\}^r
,2\mid r,
\\[8pt]
0,
&
\mathrm{otherwise}.
\end{cases}
\label{eq:constant-term-relation}
\end{align}
\end{cor}

\begin{re}
The identity
\eqref{eq:constant-term-relation} is the constant-term relation
associated with the regularized reduction formula. It is closely
related to the identity appearing in Hirose's multitangent proof of
the parity theorem for multiple zeta values; see
\cite{Hirose2025}.
\end{re}

\medskip

\noindent\textbf{Declaration of competing interests.}
The author declares that he has no known competing interests.

\medskip

\noindent\textbf{Data availability.}
No data were used or generated in this study.

\noindent\textbf{Acknowledgments.} The author would like to thank Professors Liang Xiao, Qingchun Tian, and Binyong Xie for their invaluable support and assistance throughout this research, and also extends appreciation to the School of Mathematical Sciences at Peking University for providing a pleasant working environment.


\begin{thebibliography}{99}

\bibitem{Brown2012}
F. Brown,
Mixed Tate motives over \(\mathbb Z\),
\emph{Ann. of Math. (2)} \textbf{175} (2012), no.~2,
949-976.




\bibitem{Bouillot2014}
O. Bouillot,
The algebra of multitangent functions,
\emph{J. Algebra} \textbf{410} (2014), 148-238.

\bibitem{BroadKre1997}
D.J. Broadhurst and D. Kreimer, D., Association of multiple zeta values with positive knots via Feynman diagrams up to 9 loops, Phys. Lett. B, 1997, 393(3/4), pp.\ 403-412.

\bibitem{BJ}
Jos\'e. Ignacio. Burgos Gil and Javier. Fres\'an, Multiple zeta values: from number to motives, http://javier.fresan.perso.math.cnrs.fr/mzv.pdf.





\bibitem{Euler1}
L. Euler, ``Meditationes circa singulare serierum genus'', Novi Comm. Acad. Sci. Petropolitanae, 20 (1775), 140-186


\bibitem{Goncharov2001}
A.B. Goncharov, Multiple polylogarithms and mixed Tate motives (2001), arXiv:math/0103059.

\bibitem{IKZ2006}
K. Ihara, M. Kaneko and D. Zagier,
Derivation and double shuffle relations for multiple zeta values,
\emph{Compos. Math.} \textbf{142} (2006), 307-338.

\bibitem{Hirose2025}
M. Hirose,
An explicit parity theorem for multiple zeta values via multitangent functions,
\emph{Ramanujan J.} \textbf{67} (2025), Art. 87.


\bibitem{H1992}
M.E. Hoffman, Multiple harmonic series, \emph{Pacific J.\ Math.} \textbf{152}(1992), pp.\ 275--290.

\bibitem{H1997}
M.E. Hoffman, The algebra of multiple harmonic series, \emph{J. Algebra} \textbf{194} (1997), 477–495


\bibitem{KanekoYa2018}
M. Kaneko and S. Yamamoto, A new integral-series identity of multiple zeta values and regularizations, \emph{Selecta Math.} \textbf{24}(2018), pp.\ 2499--2521.


\bibitem{KanekoTs2019}
M. Kaneko and H. Tsumura, On multiple zeta values of level two, \emph{Tsukuba J.Math.} \textbf{44-2}(2020), pp.\ 213--234.

\bibitem{LiCe2025}
 J. Li and C. Xu, Residue theorem, regularization and parity theorem, arXiv:2601.05024.






\bibitem{Todorov2014}
I. Todorov,  Polylogarithms and multizeta values in massless Feynman amplitudes. In Lie Theory and Its Applications in Physics; Dobrev, V., Ed.; Springer: Berlin, Heidelberg, 2014; Volume 111.









\bibitem{DZ1994}
D. Zagier, Values of zeta functions and their applications, First European Congress
of Mathematics, Volume II, Birkhauser, Boston, \textbf{120}(1994), pp.\ 497--512.

\bibitem{Zhao2007d}
J. Zhao, Analytic continuation of multiple polylogarithms, \emph{Anal.\ Math.} \textbf{33}(2007), pp.\ 301--323.

\bibitem{Z2016}
J. Zhao, \emph{Multiple zeta functions, multiple polylogarithms and their special values}, Series on Number
Theory and its Applications, Vol.~12, World Scientific Publishing Co. Pte. Ltd., Hackensack, NJ, 2016.



\end{thebibliography}
\end{document}